\documentclass[reqno]{amsart}
\usepackage{amsmath, amssymb, amsthm, epsfig}
\usepackage{hyperref, latexsym}
\usepackage{url}
\usepackage[mathscr]{euscript}

\usepackage{color}
\usepackage{fullpage} 
\usepackage{setspace}

\def\today{\ifcase\month\or
  January\or February\or March\or April\or May\or June\or
  July\or August\or September\or October\or November\or December\fi
  \space\number\day, \number\year}
\DeclareMathOperator{\sgn}{\mathrm{sgn}}

\DeclareMathOperator{\supp}{\mathrm{supp}}

 \newtheorem{theorem}{Theorem}
  
 \newtheorem{lemma}[theorem]{Lemma}
 \newtheorem{proposition}[theorem]{Proposition}
 \newtheorem{corollary}[theorem]{Corollary}
 \theoremstyle{definition}

 \theoremstyle{remark}

 \newcommand{\mc}{\mathcal}

 \newcommand{\U}{\mc{U}}

 \newcommand{\C}{\mathbb{C}}
 \newcommand{\R}{\mathbb{R}}
 \newcommand{\N}{\mathbb{N}}
 
 \newcommand{\Z}{\mathbb{Z}}

 \newcommand{\hh}{\tfrac12}
 \newcommand{\ds}{\text{\rm d}s}

 \newcommand{\dt}{\text{\rm d}t}

 \newcommand{\dx}{\text{\rm d}x}

 \newcommand{\dl}{\text{\rm d}\lambda}

\newcommand{\dsigma}{\text{\rm d}\sigma}

 \newcommand{\dmu}{\text{\rm d}\mu}

    \renewcommand{\d}{\text{\rm d}}

   \newcommand{\dvar}{\text{\rm d}\vartheta}

\newcommand{\ov}{\overline}
\renewcommand{\H}{\mc{H}}
\newcommand{\im}{{\rm Im}\,}
\newcommand{\re}{{\rm Re}\,}
\newcommand{\wt}{\widetilde}
\newcommand{\ms}{\mathscr}

 \newcommand{\uc}{\partial \mathbb{D}}
\newcommand{\ud}{\mathbb{D}}
\newcommand{\dbpn}{\mathcal{H}_n(P)}

\begin{document}

\title[]{Extremal problems in de Branges spaces: \\
the case of truncated and odd functions}
\author[Carneiro and Gon\c{c}alves]{Emanuel Carneiro and Felipe Gon\c{c}alves}
\date{\today}
\address{IMPA - Instituto Nacional de Matem\'{a}tica Pura e Aplicada - Estrada Dona Castorina, 110, Rio de Janeiro, RJ, Brazil 22460-320}
\email{carneiro@impa.br}
\email{lipe239@gmail.com}

\allowdisplaybreaks
\numberwithin{equation}{section}

\maketitle

%%%%%%%%  Abstract  %%%%%%%%%%%%%
\begin{abstract} In this paper we find extremal one-sided approximations of exponential type for a class of truncated and odd functions with a certain exponential subordination. These approximations optimize the $L^1(\R, |E(x)|^{-2}\dx)$-error, where $E$ is an arbitrary Hermite-Biehler entire function of bounded type in the upper half-plane. This extends the work of Holt and Vaaler \cite{HV} for the signum function. We also provide periodic analogues of these results, finding optimal one-sided approximations by trigonometric polynomials of a given degree to a class of periodic functions with exponential subordination. These extremal trigonometric polynomials optimize the $L^1(\R/\Z, \dvar)$-error, where $\vartheta$ is an arbitrary nontrivial measure on $\R/\Z$. The periodic results extend the work of Li and Vaaler \cite{LV}, who considered this problem for the sawtooth function with respect to Jacobi measures. Our techniques are based on the theory of reproducing kernel Hilbert spaces (of entire functions and of polynomials) and on the construction of suitable interpolations with nodes at the zeros of Laguerre-P\'{o}lya functions.

\end{abstract}

%%%%%%%%  Introduction  %%%%%%%%%%%%%

\section{Introduction}
\subsection{Background}
An entire function $F: \C \to \C$, not identically zero, is said to be of {\it exponential type} if
\begin{equation*}
\tau(F) = \limsup_{|z| \to \infty} |z|^{-1}\, \log |F(z)| < \infty.
\end{equation*}
In this case, the nonnegative number $\tau(F)$ is called the exponential type of $F$. We say that $F$ is {\it real entire} if $F$ restricted to $\R$ is real valued. Given a function $f:\R \to \R$, a nonnegative Borel measure $\sigma$ on $\R$, and a parameter $\delta >0$, we address here the problem of finding a pair of real entire functions $L:\C \to \C$ and $M: \C \to \C$ of exponential type at most $\delta$ such that 
\begin{equation}\label{Intro_EP1}
L(x) \leq f(x) \leq M(x)
\end{equation}
for all $x \in \R$, minimizing the integral
\begin{equation} \label{Intro_EP2}
\int_{-\infty}^{\infty} \big\{ M(x) - L(x)\big\}\, \dsigma(x).
\end{equation}

\smallskip

In the case of Lebesgue measure, this problem dates back to the work of A. Beurling in the 1930's (see \cite{V}), where the function $f(x) = \sgn(x)$ was considered. Further developments of this theory provide the solution of this extremal problem for a wide class of functions $f$ that includes, for instance, even, odd and truncated functions subject to a certain exponential or Gaussian subordination \cite{CL, CL2, CLV, CV2, CV3, DL, GV, L1, L3, L4, V}. Several applications of these extremal functions arise in analytic number theory and analysis, for instance in connection to: large sieve inequalities \cite{HV,M, V}, Erd\"{o}s-Tur\'{a}n inequalities \cite{CV2, HKW, LV,V}, Hilbert-type inequalities \cite{CL3, CLV, CV2, GV, V}, Tauberian theorems \cite{GV}, inequalities in signal processing \cite{DL}, and bounds in the theory of the Riemann zeta-function \cite{CC, CCM, CCLM, CS, Ga, GG}. Similar approximation problems are treated, for instance, in \cite{G, GL}.

\smallskip

In the case of general measures $\sigma$, the problem \eqref{Intro_EP1} - \eqref{Intro_EP2} is still vastly open. In the remarkable paper \cite{HV}, Holt and Vaaler considered the situation $f(x) = \sgn(x)$ and $\dsigma(x) = |x|^{2\nu +1}\,\dx$ with $\nu > -1$. They solved this problem (in fact, for a more general class of measures) by establishing an interesting connection with the theory of de Branges spaces of entire functions  \cite{B}. This idea was further developed in \cite{CL3} for a class of even functions $f$ with exponential subordination and in \cite{CCLM, L4} for characteristic functions of intervals, both with respect to general de Branges measures. In particular, the optimal construction in \cite{CCLM} was used to improve the existing bounds for the pair correlation of zeros of the Riemann zeta-function, under the Riemann hypothesis, extending a classical result of Gallagher \cite{Ga}.

\smallskip

The purpose of this paper is to complete the framework initiated in \cite{CL3}, where the case of even functions was treated. Here we develop an analogous extremal theory for a wide class of {\it truncated and odd functions} with exponential subordination, with respect to general de Branges measures (these are described below). In particular, this extends the work of Holt and Vaaler \cite{HV} for the signum function.

\subsection{De Branges spaces} In order to properly state our results, we need to briefly review the main concepts and terminology of the theory of Hilbert spaces of entire functions developed by L. de Branges \cite{B}. Throughout the text we denote by 
\begin{equation*}
\U = \{z \in \C; \, \im(z) >0\}
\end{equation*}
the open upper half-plane. An analytic function $F: \U \to \C$ has {\it bounded type} if it can be written as a quotient of two functions that are analytic and bounded in $\U$ (or equivalently, if $\log |F(z)|$ admits a positive harmonic majorant in $\U$). If $F: \U \to \C$ is not identically zero and has bounded type, from its Nevanlinna factorization \cite[Theorems 9 and 10]{B}, the number 
\begin{equation*}
v(F) = \limsup_{y \to \infty} \,y^{-1} \,\log |F(iy)|,
\end{equation*}
called the {\it mean type} of $F$, is finite. 

\smallskip

If $E: \C \to \C$ is entire, we define the entire function $E^*: \C \to \C$ by $E^*(z) = \ov{E(\ov{z})}$. A {\it Hermite-Biehler} function $E: \C \to \C$ is an entire function that satisfies the basic inequality
\begin{equation*}
|E^*(z)| < |E(z)|
\end{equation*}
for all $z \in \U$. If $E$ is a Hermite-Biehler function, we define the {\it de Branges space} $\H(E)$ as the space of entire functions $F:\C \to \C$ such that 
\begin{equation}\label{Intro_Def_norm_E}
\|F\|^2_{E} = \int_{-\infty}^{\infty} |F(x)|^2\, |E(x)|^{-2}\,\dx < \infty\,, 
\end{equation}
and such that $F/E$ and $F^*/E$ have bounded type in $\U$ with nonpositive mean type. This is a Hilbert space with inner product given by
\begin{equation*}
\langle F, G \rangle_E = \int_{-\infty}^{\infty} F(x)\, \ov{G(x)}\, |E(x)|^{-2}\,\dx.
\end{equation*}
The remarkable property about these spaces is that, for each $w \in \C$, the evaluation map $F \mapsto F(w)$ is a continuous linear functional. Therefore, there exists a function $K(w, \cdot) \in \H(E)$ such that 
\begin{equation*}
F(w) = \langle F,  K(w, \cdot) \rangle_E
\end{equation*}
for each $F \in \H(E)$. Such a function $K(w, z)$ is called the {\it reproducing kernel} of $\H(E)$. 

\smallskip

Associated to $E$, we define the companion functions
\begin{equation}\label{Intro_companion}
A(z) := \frac12 \big\{E(z) + E^*(z)\big\} \ \ \ {\rm and}  \ \ \ B(z) := \frac{i}{2}\big\{E(z) - E^*(z)\big\}.
\end{equation}
Note that $A$ and $B$ are real entire functions such that $E(z) = A(z) - iB(z)$. The reproducing kernel is given by \cite[Theorem 19]{B}
\begin{align}\label{Intro_K_w_z}
K(w,z)  = \frac{E(z)E^*(\ov{w}) - E^*(z)E(\ov{w})}{2\pi i (\ov{w}-z)} = \frac{B(z)A(\ov{w}) - A(z)B(\ov{w})}{\pi (z - \ov{w})},
\end{align}
and when $z = \ov{w}$ we have
\begin{equation}\label{Intro_K_z_z}
K(\ov{z}, z) = \frac{B'(z)A(z) - A'(z)B(z)}{\pi}.
\end{equation}
From the reproducing kernel property we have
\begin{align*}
K(w,w) = \langle K(w, \cdot),  K(w, \cdot) \rangle_E = \|K(w, \cdot)\|^2_E \geq 0\,,
\end{align*}
and one can easily show that $K(w,w) =0$ if and only if $w \in \R$ and $E(w) =0$ (see for instance \cite[Lemma 11]{HV} or \cite[Problem 45]{B}).

\subsection{Main results} For our purposes we let $E$ be a Hermite-Biehler function of bounded type in $\U$. In this case, a classical result of M. G. Krein (see \cite{K} or \cite[Lemma 9]{HV}) guarantees that $E$ has exponential type and $\tau(E) = v(E)$. Moreover, an entire function $F$ belongs to $\H(E)$ if and only if it has exponential type at most $\tau(E)$ and satisfies \eqref{Intro_Def_norm_E} (see \cite[Lemma 12]{HV}). 

\smallskip

Let $\mu$ be a (locally finite) signed Borel measure on $\R$ satisfying the following two properties:

\begin{enumerate}

\item[(H1)] The measure $\mu$ has support bounded by below.

\smallskip

\item[(H2)] The right-continuous distribution function associated to this measure (that we keep calling $\mu$, with a slight abuse of notation), defined by $\mu (x):= \mu((-\infty, x])$, verifies
\begin{equation}\label{Hyp1}
0 \leq \mu(x) \leq 1
\end{equation}
for all $x \in \R$.
\end{enumerate}

\noindent In some instances we require a third property:

\begin{enumerate}

\item[(H3)] The average value of the distribution function $\mu$ is $1$, i.e.
\begin{equation}\label{Hyp2}
\lim_{y \to \infty} \frac{1}{y}\int_{-\infty}^{y} \mu(x)\,\dx = 1.
\end{equation}

\end{enumerate}

\smallskip

\noindent We remark that the constant $1$ appearing on the right-hand sides of \eqref{Hyp1} and \eqref{Hyp2} could be replaced by any constant $C>0$. For simplicity, we normalize the measure (by dilating) to work with $C=1$.  Observe that any probability measure $\mu$ on $\R$ satisfying (H1) automatically satisfies (H2) and (H3). Measures like $\dmu(\lambda) =  \chi_{(0,\infty)}(\lambda)\,\sin a\lambda\, \dl$, for $a>0$, which were considered by Littmann and Spanier in \cite{LS} (giving the truncated and odd Poisson kernels in the construction below), satisfy (H1) - (H2) but not (H3).

\smallskip

Let $\mu$ be a signed Borel measure on $\R$ satisfying (H1) - (H2). We define the function $f_{\mu}$, the truncated Laplace transform of this measure, by 
\begin{equation}\label{Intro_Def_f_mu}
f_{\mu}(z) = \left\{
\begin{array}{lc}
\displaystyle\int_{-\infty}^{\infty} e^{-\lambda z}\,\dmu(\lambda), & {\rm if} \ \re(z) >0; \medskip \\
0, & {\rm if} \ \re(z) \leq 0.
\end{array}
\right.
\end{equation}
Observe that $f_{\mu}$ is a well-defined analytic function in $\re(z) >0$ since
\begin{equation}\label{Intro_int_parts}
f_{\mu}(z) = \int_{-\infty}^{\infty} e^{-\lambda z}\,\dmu(\lambda) = \int_{-\infty}^{\infty} z e^{-\lambda z} \mu(\lambda)\,\dl,
\end{equation}
where we have used integration by parts. If we write 
$$\mu^{(-1)}(y) := \int_{-\infty}^y \mu(x)\,\dx,$$ 
under the additional condition (H3) we find that (below we let $\supp(\mu) \subset (a,\infty)$)
\begin{align}\label{Intro_H3}
\begin{split}
f_{\mu}(0^+) & = \lim_{x \to 0^+} f_{\mu}(x) = \lim_{x \to 0^+} \int_{a}^{\infty} x e^{-\lambda x} \mu(\lambda)\,\dl =  \lim_{x \to 0^+} \int_{a}^{\infty} x^2 e^{-\lambda x} \mu^{(-1)}(\lambda)\,\dl\\
& = \lim_{x \to 0^+} \int_{ax}^{\infty}  x e^{-t} \mu^{(-1)}(t/x)\, \dt= \int_{0}^{\infty}  t e^{-t}\, \dt \\
& = 1,
\end{split}
\end{align}
by dominated convergence. Our first result is the following.

\begin{theorem}\label{thm1}
Let $E$ be a Hermite-Biehler function of bounded type in $\U$ such that $E(0) \neq 0$. Let $\mu$ be a signed Borel measure on $\R$ satisfying  {\rm (H1) - (H2) - (H3)}. Assume that $\supp(\mu) \subset [-2\tau(E), \infty)$ and let $f_{\mu}$ be defined by \eqref{Intro_Def_f_mu}. If $L:\C \to \C$ and $M:\C \to \C$ are real entire functions of exponential type at most $2 \tau(E)$ such that
\begin{equation}\label{Intro_eq_L_M}
L(x) \leq f_{\mu}(x) \leq M(x) 
\end{equation}
for all $x \in \R$, then 
\begin{equation}\label{Intro_answer_1}
\int_{-\infty}^{\infty} \big\{M(x) - L(x)\big\}\, |E(x)|^{-2}\,\dx \geq \frac{1}{K(0,0)}. \smallskip
\end{equation} 
Moreover, there is a unique pair of real entire functions $L_{\mu}:\C \to \C$ and $M_{\mu}:\C \to \C$ of exponential type at most $2\tau(E)$ satisfying \eqref{Intro_eq_L_M} for which the equality in \eqref{Intro_answer_1} holds.
\end{theorem}

Our second result is the analogous of Theorem \ref{thm1} for the odd function 
\begin{equation}\label{Intro_Def_f_mu_o}
\widetilde{f}_{\mu}(z) := f_{\mu}(z) - f_{\mu}(-z).
\end{equation}
Note that if $\mu$ is the Dirac delta measure we have $\wt{f}_{\mu}(x) = \sgn(x)$.

\begin{theorem}\label{thm2}
Let $E$ be a Hermite-Biehler function of bounded type in $\U$ such that $E(0) \neq 0$. Let $\mu$ be a signed Borel measure on $\R$ satisfying  {\rm (H1) - (H2) - (H3)}. Assume that $\supp(\mu) \subset [-2\tau(E), \infty)$ and let $\wt{f}_{\mu}$ be defined by \eqref{Intro_Def_f_mu_o}. If $L:\C \to \C$ and $M:\C \to \C$ are real entire functions of exponential type at most $2 \tau(E)$ such that
\begin{equation}\label{Intro_eq_L_M_o}
L(x) \leq \wt{f}_{\mu}(x) \leq M(x) 
\end{equation}
for all $x \in \R$, then 
\begin{equation}\label{Intro_answer_1_o}
\int_{-\infty}^{\infty} \big\{M(x) - L(x)\big\}\, |E(x)|^{-2}\,\dx \geq \frac{2}{K(0,0)}. \smallskip
\end{equation} 
Moreover, there is a unique pair of real entire functions $\wt{L}_{\mu}:\C \to \C$ and $\wt{M}_{\mu}:\C \to \C$ of exponential type at most $2\tau(E)$ satisfying \eqref{Intro_eq_L_M_o} for which the equality in \eqref{Intro_answer_1_o} holds.
\end{theorem}

\noindent{\sc Remark 1:} There is no loss of generality in assuming $E(0) \neq 0$ and $\supp(\mu) \subset [-2\tau(E), \infty)$ in Theorems \ref{thm1} and \ref{thm2}. In fact, since $f_{\mu}(x)$ and $\wt{f}_{\mu}(x)$ are discontinuous at $x =0$, if $E(0) = 0$ the integrals on the left-hand sides of \eqref{Intro_answer_1} and \eqref{Intro_answer_1_o} always diverge. Given $\varepsilon >0$, if the set $\{x \in \R;\ \mu(x) > 0\} \cap  (-\infty, -2\tau(E) - \varepsilon)$ has nonzero Lebesgue measure, we find by \eqref{Intro_int_parts} that $f_{\mu}(x) \geq C_{\varepsilon}\, x\,e^{(2\tau(E) + \varepsilon)x}$ for $x >0$, and there are no entire functions $L$ and $M$ of exponential type at most $2\tau(E)$ satisfying \eqref{Intro_eq_L_M} or \eqref{Intro_eq_L_M_o}. 

\smallskip

\noindent{\sc Remark 2:} The minorant problem for $f_{\mu}$ can be solved without the hypothesis (H3). We give the details in Corollary \ref{Cor7} below.

\smallskip

\noindent{\sc Remark 3:} Note that we are allowing the measure $\mu$ to have part of its support on the negative axis. In principle, our function $f_{\mu}(x)$ could increase exponentially as $x \to \infty$ and does not necessarily belong to $L^1(\R, |E(x)|^{-2}\,\dx)$ (the same holds for $L$ and $M$). When $f_{\mu} \in L^1(\R, |E(x)|^{-2}\,\dx)$ (resp. $\wt{f}_{\mu} \in L^1(\R, |E(x)|^{-2}\,\dx)$) it is possible to determine the corresponding optimal values of 
\begin{equation*}
\int_{-\infty}^{\infty} M(x) \,|E(x)|^{-2}\,\dx \ \ \ \ {\rm and} \ \ \ \ \int_{-\infty}^{\infty} L(x) \,|E(x)|^{-2}\,\dx
\end{equation*}
separately. This is detailed in Corollaries \ref{Prop7_sep} and \ref{Prop8_sep} below. 

\smallskip

We use two main tools in the proofs of Theorems \ref{thm1} and \ref{thm2}. The first is a basic Cauchy-Schwarz inequality in the Hilbert space $\H(E)$ that shows that the optimal choice for  $M(z)-L(z)$ must be the square of the reproducing kernel at the origin (divided by a constant). The second tool, used to show the existence of such optimal majorants and minorants, is the construction of suitable entire functions that interpolate $f_{\mu}$ at the zeros of a given Laguerre-P\'{o}lya function. The latter is detailed in Section \ref{Interpolation_sec} and extends the construction of Holt and Vaaler \cite[Section 2]{HV}, that was tailored specifically for the signum function.

\subsection{A class of homogeneous spaces} There is a variety of examples of de Branges spaces \cite[Chapter 3]{B} for which Theorems \ref{thm1} and \ref{thm2} can be directly applied. A basic example would be the classical Paley-Wiener space $\H(E)$, when $E(z) = e^{-i\tau z}$ with $\tau >0$,  which gives us extremal functions of exponential type at most $2\tau$ with respect to Lebesgue measure (this extends the classical work of Graham and Vaaler \cite{GV}). Another interesting family arises in the discussion of \cite[Section 5]{HV}. In the terminology of de Branges \cite[Section 50]{B}, these are examples of homogeneous spaces, and we briefly review their construction below.

\smallskip

Let $\nu > -1$ be a parameter and consider the real entire functions $A_\nu$ and $B_\nu$ given by
\begin{equation*}
A_{\nu}(z) = \sum_{n=0}^{\infty} \frac{(-1)^n \big(\tfrac12 z\big)^{2n}}{n!(\nu +1)(\nu +2)\ldots(\nu+n)} = \Gamma(\nu +1) \left(\tfrac12 z\right)^{-\nu} J_{\nu}(z)
\end{equation*}
and
\begin{equation*}
B_{\nu}(z) = \sum_{n=0}^{\infty} \frac{(-1)^n \big(\tfrac12 z\big)^{2n+1}}{n!(\nu +1)(\nu +2)\ldots(\nu+n+1)} = \Gamma(\nu +1) \left(\tfrac12 z\right)^{-\nu} J_{\nu+1}(z),
\end{equation*}
where $J_\nu$ denotes the classical Bessel function of the first kind. If we write
\begin{equation*}
E_\nu(z) = A_\nu(z) - iB_\nu(z),
\end{equation*}
then the function $E_\nu$ is a Hermite-Biehler function of bounded type in $\U$ with exponential type $\tau(E_\nu) = 1$ and no real zeros. Observe that when $\nu = -1/2$ we have simply $A_{-1/2}(z) = \cos z$ and $B_{-1/2}(z) = \sin z$. For a general $\nu > -1$, there are positive constants $a_\nu$ and $b_\nu$ such that 
\begin{equation}\label{Intro_homog_eq1}
a_\nu |x|^{2\nu+1} \le |E_{\nu}(x)|^{-2} \le b_\nu |x|^{2\nu+1}
\end{equation}
for all $x \in \R$ with $|x|\geq1$. For each $F\in\H(E_\nu)$ we have the remarkable identity 
\begin{align}\label{Intro_homog_eq2}
\int_{-\infty}^\infty |F(x)|^{2}\,|E_{\nu}(x)|^{-2}\, \dx = c_\nu \int_{-\infty}^\infty |F(x)|^2 \,|x|^{2\nu+1} \,\dx\,,
\end{align}
with $c_\nu = \pi \,2^{-2\nu-1}\, \Gamma(\nu+1)^{-2}$, and from \eqref{Intro_homog_eq1} and \eqref{Intro_homog_eq2} we see that $F \in \H(E_\nu)$ if and only if $F$ has exponential type at most $1$ and either side of \eqref{Intro_homog_eq2} is finite. Identity \eqref{Intro_homog_eq2} makes $\H(E_{\nu})$ the suitable de Branges space to treat the extremal problem \eqref{Intro_EP1} - \eqref{Intro_EP2} for the power measure $\dsigma(x) = |x|^{2\nu +1}\,\dx$. In order to do so, we define 
\begin{equation*}
\Delta_\nu(\delta, \mu) = \inf \int_{-\infty}^{\infty} \big\{ M(x) - L(x)\big\} \,|x|^{2\nu+1}\,\dx\,,
\end{equation*}
where the infimum is taken over all pairs of real entire functions $L:\C \to \C$ and $M:\C \to \C$ of exponential type at most $\delta$ such that $L(x) \leq f_{\mu}(x) \leq M(x)$ for all $x \in \R$. If there is no such a pair we set $\Delta_\nu(\delta, \mu) = \infty$. Define $\wt{\Delta}_\nu(\delta, \mu)$ considering the analogous extremal problem for the odd function $\wt{f}_{\mu}$. The following result follows from Theorems \ref{thm1} and \ref{thm2}.
\begin{theorem}\label{thm3}
Let $\nu > -1$ and $\delta >0$. Let $\mu$ be a signed Borel measure on $\R$ satisfying  {\rm (H1) - (H2) - (H3)}, and let $f_{\mu}$ be defined by \eqref{Intro_Def_f_mu} {\rm (}resp. $\wt{f}_{\mu}$ be defined by \eqref{Intro_Def_f_mu_o}\rm{)}. We have
\begin{equation}\label{Intro_thm3_1}
\Delta_\nu(\delta, \mu) = \left\{
\begin{array}{ll}
 \Gamma(\nu+1)\,\Gamma(\nu+2) \left(\frac{4}{\delta}\right)^{2\nu +2},&  \ \ {\rm if}\, \supp(\mu) \subset [-\delta, \infty);\smallskip\\
\infty,& \ \ {\rm otherwise};
\end{array}
\right.\\
\end{equation}
and
\begin{equation}\label{Intro_thm3_2}
\wt{\Delta}_\nu(\delta, \mu) = \left\{
\begin{array}{ll}
 2\,\Gamma(\nu+1)\,\Gamma(\nu+2) \left(\frac{4}{\delta}\right)^{2\nu +2},&  \ \ {\rm if}\, \supp(\mu) \subset [-\delta, \infty);\smallskip\\
\infty,& \ \ {\rm otherwise}.
\end{array}
\right.\\
\end{equation}
If $\Delta_\nu(\delta, \mu)$ {\rm (}resp. $\wt{\Delta}_\nu(\delta, \mu)${\rm)} is finite, there exists a unique pair of corresponding extremal functions. 
\end{theorem}
\begin{proof} To see why Theorem \ref{thm3} is indeed a consequence of Theorems \ref{thm1} and \ref{thm2} we proceed as follows. For $\kappa >0$, we consider the measure $\mu_{\kappa}$ defined by $\mu_{\kappa}(\Omega) = \mu(k\Omega)$, where $\Omega$ is any Borel measurable set and $\kappa \Omega = \{\kappa\lambda; \ \lambda \in \Omega\}$. A simple dilation argument shows that
\begin{equation*}
\Delta_\nu(\delta, \mu) = \kappa^{2\nu +2}\,\Delta_\nu\big(\kappa\delta, \mu_{\kappa^{-1}}\big) \ \ \ {\rm and} \ \ \ \wt{\Delta}_\nu(\delta, \mu) = \kappa^{2\nu +2}\,\wt{\Delta}_\nu\big(\kappa\delta, \mu_{\kappa^{-1}}\big),
\end{equation*}
and we can reduce matters to the case $\delta =2$. Now let $L$ and $M$ be a pair of real entire functions of exponential type at most $2$ such that $L(x) \leq f_{\mu}(x) \leq M(x)$ for all $x \in \R$, and such that $(M-L) \in L^1(\R, |x|^{2\nu+1}\,\dx)$. By \eqref{Intro_homog_eq1} we have that $(M-L) \in L^1(\R, |E_{\nu}(x)|^{-2}\,\dx)$. Since $(M-L)$ is nonnegative on $\R$, according to \cite[Theorem 15]{HV} (see also \cite[Lemma 14]{CL3}) we can write $M(z) - L(z) = U(z)U^*(z)$ with $U \in \H(E_\nu)$. Therefore, by identity \eqref{Intro_homog_eq2} and Theorem \ref{thm1}, we have
\begin{align*}
\int_{-\infty}^{\infty}  \big\{ M(x) & - L(x)\big\} \, |x|^{2\nu+1}\,\dx   = \int_{-\infty}^{\infty} |U(x)|^2\, |x|^{2\nu+1}\,\dx = c_\nu^{-1}  \int_{-\infty}^{\infty} |U(x)|^2\, |E_\nu(x)|^{-2}\,\dx\\
& = c_\nu^{-1}  \int_{-\infty}^{\infty} \big\{ M(x) - L(x)\big\} \, |E_\nu(x)|^{-2}\,\dx \geq c_\nu^{-1}\, K_\nu(0,0)^{-1},
\end{align*}
where $c_\nu = \pi \,2^{-2\nu-1}\, \Gamma(\nu+1)^{-2}$ and 
\begin{equation*}
K_\nu(0,0) = \frac{B_\nu'(0) A_\nu(0)}{\pi} = \frac{1}{2 \pi (\nu+1)}.
\end{equation*}
This establishes \eqref{Intro_thm3_1}. A similar argument using Theorem \ref{thm2} gives \eqref{Intro_thm3_2}.

\end{proof}

As illustrated in the argument above, in order to use the general machinery of Theorems \ref{thm1} and \ref{thm2} to solve the extremal problem \eqref{Intro_EP1} - \eqref{Intro_EP2} for a given measure $\sigma$, one has to first construct an appropriate de Branges space $\H(E)$ that is isometrically contained in $L^2(\R, \dsigma)$. In particular, this construction was carried out in \cite{CCLM} for the measure
\begin{equation*}
\dsigma(x) = \left\{1 - \left(\frac{\sin \pi x}{\pi x}\right)^2 \right\}\dx,
\end{equation*}
that appears in connection to Montgomery's formula and the pair correlation of zeros of the Riemann zeta-function (see \cite{M2}), and in \cite{LS} for the measure 
\begin{equation*}
\dsigma(x) = (x^2 + a^2)\,\dx,
\end{equation*}
where $a \geq 0$, that appears in connection to extremal problems with prescribed vanishing conditions.

\subsection{Periodic analogues} \label{par_Per}

In Section \ref{Per_Analogues} we consider the periodic version of this extremal problem. Throughout the paper we write $e(z) = e^{2\pi i z}$ for $z \in \C$. A trigonometric polynomial of degree at most $N$ is an entire function of the form
$$\ms{W}(z) = \sum_{k=-N}^{N} a_k \,e(kz),$$
where $a_k \in \C$. We say that $\ms{W}$ is a {\it real trigonometric polynomial} if $\ms{W}(z)$ is real for $z$ real.
Given a periodic function $\ms{F}:\R/\Z \to \R$, a probability measure $\vartheta$ on  $\R/\Z$ and a degree $N \in \Z^+$, we address in Section \ref{Per_Analogues} the problem of finding a pair of real trigonometric polynomials $\ms{L}:\C \to \C$ and $\ms{M}: \C \to \C$ of degree at most $N$ such that  
\begin{equation}\label{Trig_EP1}
\ms{L}(x) \leq \ms{F}(x) \leq \ms{M}(x)
\end{equation}
for all $x \in \R/\Z$, minimizing the integral
\begin{equation} \label{Trig_EP2}
\int_{\R/\Z} \big\{ \ms{M}(x) - \ms{L}(x)\big\}\, \dvar(x).
\end{equation}
When $\vartheta$ is the Lebesgue measure, this problem was considered, for instance, in \cite{Car, CV2, LV, V} in connection to discrepancy inequalities of Erd\"{o}s-Tur\'{a}n type. For general even measures $\vartheta$, the case of even periodic functions with exponential subordination was considered in \cite{BV, CL3}. In \cite{LV}, Li and Vaaler solved this extremal problem for the sawtooth function 
$$\psi(x) = 
\left\{
\begin{array}{lc}
x - \lfloor x \rfloor - \hh&, \ {\rm if} \ x \notin \Z;\\
0&, \ {\rm if} \ x \in \Z;
\end{array}
\right.
$$ 
with respect to the Jacobi measures. The purpose of Section \ref{Per_Analogues} is to extend the work \cite{LV}, solving this problem for a general class of functions with exponential subordination (which are the periodizations of our functions $f_{\mu}$ and $\wt{f}_{\mu}$, including the sawtooth function as a particular case) with respect to arbitrary nontrivial probability measures $\vartheta$ (we say that $\vartheta$ is trivial if it has support on a finite number of points). The solution of this periodic extremal problem is connected to the theory of reproducing kernel Hilbert spaces of polynomials and the theory of orthogonal polynomials in the unit circle.

%%%%%%%%%%%%%%%%%%%%%%%%%%%%%%%%%%%%%%%%%%%%%%%%%%%%%%%%%%%%%%%%%%%%%%%%%%%%%%%%%%%%%%%%%%%%%%%%%%%%%%%%%%%%%%%%%%%%%%%%%%%%%%%%%%%%%%%%%%%%%%%%%%%%%%%%%%%%%%%%%%%%%%%%

\section{Interpolation tools}\label{Interpolation_sec}

\subsection{Laplace transforms and Laguerre-P\'{o}lya functions} In this subsection we review some basic facts concerning Laguerre-P\'{o}lya functions and the representation of their inverses as Laplace transforms as in \cite[Chapters II to V]{HW}. The selected material we need is already well organized in \cite[Section 2]{CL3} and we follow closely their notation.

\smallskip

We say that an entire function $F: \C \to \C$ belongs to the {\it Laguerre-P\'{o}lya class} if it has only real zeros and its Hadamard factorization is given by
\begin{equation}\label{hadamardproduct}
F(z)=\frac{F^{(r)}(0)}{r!}\,z^r\, e^{-az^2+bz}\,\prod_{j=1}^\infty\Big(1-\frac{z}{x_j}\Big)e^{z/x_j},
\end{equation}
where $r\in\Z^{+}$, $a, b, x_j \in \R$, with $a \geq 0$, $x_j \neq 0$ and $\sum_{j=1}^\infty x_j^{-2}<\infty$ (with the appropriate change of notation in case of a finite number of zeros). Such functions are the uniform limits (in compact sets) of polynomials with only real zeros. We say that a Laguerre-P\'{o}lya function $F$ represented by (\ref{hadamardproduct}) has finite degree $\mc{N} = \mc{N}(F)$ when $a=0$ and $F$ has exactly $\mc{N}$ zeros counted with multiplicity. Otherwise we set $\mc{N}(F) = \infty$. 

\smallskip

If $F$ is a Laguerre-P\'{o}lya function with $\mc{N}(F) \geq 2$, and $c\in \R$ is such that $F(c) \neq 0$, we henceforth denote by $g_c$ the frequency function given by 
\begin{equation}\label{gc-def}
g_c(t) = \frac{1}{2\pi i } \int_{c-i\infty}^{c+i\infty} \frac{e^{ts}}{F(s)}\, \ds.
\end{equation}
Observe that the integral in \eqref{gc-def} is absolutely convergent since the condition $\mc{N}(F) \geq 2$ implies that $1/|F(c + iy)| = O(|y|^{-2})$ as $|y| \to \infty$. If $(\tau_1,\tau_2) \subset \R$ is the largest open interval containing no zeros of $F$ such that $c \in (\tau_1,\tau_2)$, the residue theorem implies that $g_c = g_d$ for any $d \in (\tau_1, \tau_2)$. Moreover, the Laplace transform representation 
\begin{equation}\label{gc-trafo}
\frac{1}{F(z)} = \int_{-\infty}^\infty g_c(t)\, e^{-tz}\, \,\dt
\end{equation}
holds in the strip $\tau_1<\re(z)<\tau_2$ (the integral in \eqref{gc-trafo} is in fact absolutely convergent due to Lemma \ref{qualitative_g_c} below). If $\mc{N}(F) = 0$ or $1$, we can still represent $F(z)^{-1}$ as a Laplace transform on vertical strips. In fact, if $\mc{N}(F) =1$, we let $\tau$ be the zero of $F$, written in the form \eqref{hadamardproduct}. If $\tau = 0$ then \eqref{gc-trafo} holds with
\begin{equation}\label{Def_g_c_N_1_0}
g_c(t) = \left\{
\begin{array}{ll}
\vspace{0.2cm}
F'(0)^{-1}\, \chi_{(b,\infty)}(t), & \ {\rm for}\ c>0;\\
-F'(0)^{-1}\, \chi_{(-\infty,b)}(t),& \ {\rm for}\ c<0.\\
\end{array}
\right.
\end{equation}
If $\tau \neq 0$ then \eqref{gc-trafo} holds with
\begin{equation}\label{Def_g_c_N_1_1}
g_c(t) = \left\{
\begin{array}{ll}
\vspace{0.2cm}
-\tau \, F(0)^{-1}\, e^{\tau(t - b) -1}\chi_{(b + \tau^{-1},\infty)}(t), & \ {\rm for}\ c>\tau;\\
\tau \, F(0)^{-1}\, e^{\tau(t - b) -1}\chi_{(-\infty,b + \tau^{-1})}(t),& \ {\rm for}\ c<\tau.\\
\end{array}
\right.
\end{equation}
If $\mc{N}(F) = 0$ then \eqref{gc-trafo} holds with 
\begin{equation*}
g_c(t) = F(0)^{-1}\,\delta(t-b),
\end{equation*}
for any $c \in \R$, where $\delta$ denotes the Dirac delta measure.

\smallskip

The fundamental tool for the development of our interpolation theory in this section is the precise qualitative knowledge of the frequency functions $g_c$. This is extensively discussed in \cite[Chapters II to V]{HW} and we collect the relevant facts for our purposes in the next lemma.

\begin{lemma}\label{qualitative_g_c} 
Let $F$ be a Laguerre-P\'olya function of degree $\mc{N}\ge 2$ and let $g_c$ be defined by \eqref{gc-def}, where $c \in \R$ and $F(c) \neq 0$. The following propositions hold:
\begin{enumerate} 
\item [(i)] The function $g_c \in C^{\mc{N}-2}(\R)$ and is real valued.
 
 \smallskip
 
\item[(ii)] The function $g_c$ is of one sign, and its sign equals the sign of $F(c)$.

\smallskip

\item[(iii)] If $(\tau_1,\tau_2) \subset \R$ is the largest open interval containing no zeros of $F$ such that $c \in (\tau_1,\tau_2)$, then for any $\tau \in (\tau_1,\tau_2)$ we have the following estimate
\begin{align}\label{Lem4_growth}
\big|g_c^{(n)}(t)\big| \ll_{\tau,n} e^{\tau t} \ \ \  \forall \,t \in \R,
\end{align}
where $0 \leq n \leq \mc{N}-2$.
\end{enumerate}
\end{lemma}

\begin{proof}
Parts (i) and (ii) follow from \cite[Chapter IV, Theorems 5.1 and 5.3]{HW}. Part (iii) follows from  \cite[Chapter II, Theorem 8.2 and Chapter V, Theorem 2.1]{HW}.
\end{proof}

\subsection{Interpolating $f_{\mu}$ at the zeros of Laguerre-P\'{o}lya functions} In this subsection we construct suitable entire functions that interpolate our $f_{\mu}$ at the zeros of a given Laguerre-P\'{o}lya function. In order to accomplish this, we make use of the representation in \eqref{gc-trafo} and Lemma \ref{qualitative_g_c}. The material in this subsection extends the classical work of Graham and Vaaler in \cite[Section 3]{GV}, where this construction was achieved for the particular function $F(x) = (\sin \pi x)^2$ of Laguerre-P\'{o}lya class.

\smallskip

If $F$ is a Laguerre-P\'{o}lya function, we henceforth denote by $\alpha_F$ the smallest positive zero of $F$ (if no such zero exists, we set $\alpha_F = \infty$). Let $g = g_{\alpha_F/2}$ (if $\alpha_F = \infty$ take $g = g_1$). If $\mu$ is a signed Borel measure on $\R$ satisfying (H1) - (H2), it is clear that the function 
\begin{equation}\label{Sec2_equ_conv}
g*\dmu(t) = \int_{-\infty}^{\infty} g(t - \lambda)\,\dmu(\lambda) = \int_{-\infty}^{\infty} g'(t - \lambda)\,\mu(\lambda)\,\dl = g'*\mu(t)
\end{equation}
satisfies the same growth conditions as in \eqref{Lem4_growth} for $\tau \in (0,\alpha_F)$, for $0 \leq n \leq \mc{N} - 3$, with the implied constants now depending also on $\mu$. We are now in position to define the building blocks of our interpolation.

\begin{proposition}\label{Prop5_int}
Let $F$ be a Laguerre-P\'{o}lya function with $\mc{N}(F) \geq 2$. Let $g=g_{\alpha_F/2}$ and assume that $F(\alpha_F/2)>0$ $($in case $\alpha_F= +\infty$, let $g = g_1$ and assume $F(1) >0 {\rm )}$. Let $\mu$ be a signed Borel measure on $\R$ satisfying {\rm (H1) - (H2)}, and let $f_{\mu}$ be defined by \eqref{Intro_Def_f_mu}. Define
\begin{align}
\mc{A}_1(F,\mu,z)  &= F(z) \int_{-\infty}^0 g*\dmu(t)\, e^{-tz}\, \dt\ \ \text{\rm for }\,\re(z)<\alpha_F,\label{Def_A_1}\\
\mc{A}_2(F,\mu,z) &=  f_{\mu}(z) - F(z) \int_{0}^\infty g*\dmu(t)\, e^{-tz} \,\dt\ \ \text{ \rm for }\, \re(z)>0.\label{Def_A_2}
\end{align}
Then $z\mapsto \mc{A}_1(F,\mu,z)$ is analytic in $\re(z)<\alpha_F$, $z\mapsto \mc{A}_2(F,\mu,z)$ is analytic in $\re(z)>0$, and these functions are restrictions of an entire function, which we will denote by $\mc{A}(F,\mu,z)$. Moreover, if  $supp(\mu)\subset[-\tau,\infty)$, there exists $c>0$ so that 
\begin{align}\label{al-growth}
|\mc{A}(F,\mu,z)| \le c\,\big( |z|\,e^{\tau x}\,\chi_{(0,\infty)}(x) + |F(z)|\big)
\end{align}
for all $z=x+iy\in\C$, and 
\begin{align}\label{zeros-a}
\mc{A}(F,\mu,\xi)=f_{\mu}(\xi)
\end{align} 
for all $\xi \in \R$ with $F(\xi)=0$.
\end{proposition}

\begin{proof}
We have already noted in \eqref{Intro_int_parts} that $z \mapsto f_{\mu}(z)$ is analytic in $\re(z) >0$ when $\mu$ satisfies (H1) - (H2). If $\mc{N}(F) \geq 3$, from \eqref{Sec2_equ_conv} and Lemma \ref{qualitative_g_c} (iii) we see that the integrals on the right-hand sides of  \eqref{Def_A_1} and \eqref{Def_A_2} converge absolutely and define analytic functions in the stated half-planes.  If $\mc{N}(F) = 2$, it can be verified directly that $g$ is continuous and $C^1$ by parts, and that the function $g'$ thus obtained has at most one discontinuity and still satisfies the growth condition \eqref{Lem4_growth}. Therefore \eqref{Sec2_equ_conv} holds and, as before, this suffices to establish the absolute convergence and analiticity of \eqref{Def_A_1} and \eqref{Def_A_2} in the stated half-planes.

\smallskip

Now let $0 < x < \alpha_F$. Using \eqref{Sec2_equ_conv}, \eqref{gc-trafo} and \eqref{Intro_int_parts} we get
\begin{align*}
\mc{A}_1(F,\mu,x)   - \mc{A}_2(F,\mu,x) & = -f_{\mu}(x)  + F(x) \int_{-\infty}^{\infty} g'*\mu(t)\, e^{-tx}\, \dt\\
& = -f_{\mu}(x)  + F(x) \int_{-\infty}^{\infty} \int_{-\infty}^{\infty} g'(t-\lambda) \,\mu(\lambda)\,e^{-tx}\,\dl\,\dt\\
& = -f_{\mu}(x)  + F(x) \int_{-\infty}^{\infty} \left(\int_{-\infty}^{\infty} g'(t-\lambda) \,e^{-tx}\,\dt\right)\, \mu(\lambda)\,\dl\\
& = -f_{\mu}(x)  + F(x) \int_{-\infty}^{\infty} \left(\int_{-\infty}^{\infty} g'(s) \,e^{-sx}\,\ds\right) e^{-\lambda x}\,\mu(\lambda)\,\dl\\
& = -f_{\mu}(x)  + \int_{-\infty}^{\infty} x \,e^{-\lambda x}\,\mu(\lambda)\,\dl\\
& = 0.
\end{align*}
This implies that $\mc{A}_1(F,\mu,z) = \mc{A}_2(F,\mu,z)$ in the strip $0 < \re(z) < \alpha_F$. Hence, $z\mapsto \mc{A}_1(F, \mu,z)$ and $z\mapsto \mc{A}_2(F, \mu,z)$ are analytic continuations of each other and this defines the entire function $z \mapsto \mc{A}(F, \lambda, z)$. The integral representations for $\mc{A}$ and \eqref{Sec2_equ_conv} imply, for $\re(z)\le \alpha_F/2$, that
\begin{align}\label{Sec_ILP_eq1}
\begin{split}
|\mc{A}(F, \mu,z)|& \le |F(z)| \int_{-\infty}^0 |g'|*\mu (t) \,e^{-t\,\re(z) } \,\dt\\
&  \le |F(z)| \int_{-\infty}^0 |g'|*\mu(t) \,e^{-t \,\alpha_F/2} \,\dt\,,
\end{split}
\end{align}
while for $\re(z)\ge \alpha_F/2$ we have
\begin{equation}\label{Sec_ILP_eq2}
|\mc{A}(F, \mu,z)|\le |f_{\mu}(z)| +|F(z)| \int_0^\infty |g'|*\mu(t) \,e^{-t \,\alpha_F/2} \,\dt.
\end{equation}
Since $\supp(\mu) \subset [-\tau, \infty)$ we use \eqref{Intro_int_parts} and (H2) to obtain, for $\re(z)\ge \alpha_F/2$,  
\begin{align}\label{Sec_ILP_eq3}
\begin{split}
|f_{\mu}(z)| & =  \left|\int_{-\tau}^{\infty} z \, e^{-\lambda z}\,\mu(\lambda)\,\dl\right| \leq |z| \left|\int_{-\tau}^{\infty}  e^{-\lambda \,\re(z)}\,\dl\right| \\
& = \frac{|z|}{\re(z)}\, e^{\tau \, \re(z)} \leq \frac{2|z|}{\alpha_F}\, e^{\tau \, \re(z)}.
\end{split}
\end{align}
Estimates \eqref{Sec_ILP_eq1}, \eqref{Sec_ILP_eq2} and \eqref{Sec_ILP_eq3} plainly verify \eqref{al-growth}. The remaining identity \eqref{zeros-a} follows from the definition of $\mc{A}$.
\end{proof}

\begin{proposition} \label{M-L-construction}
Let $F$ be a Laguerre-P\'{o}lya function that has a double zero at the origin. Let $g= g_{\alpha_F/2}$ and assume that $F(\alpha_F/2)>0$ $($in case $\alpha_F = +\infty$, let $g = g_1$ and assume $F(1) >0$ \!$)$. Let $\mu$ be a signed Borel measure on $\R$ satisfying {\rm (H1) - (H2)}, and let $f_{\mu}$ be defined by \eqref{Intro_Def_f_mu}. With $z \mapsto \mc{A}(F,\mu,z)$ defined by Proposition \ref{Prop5_int}, consider the entire functions $z \mapsto L(F,\mu,z)$ and  $z \mapsto M(F,\mu,z)$ defined by
\begin{equation}\label{L-def}
L(F,\mu,z) = \mc{A}(F,\mu,z) + g*\dmu(0) \frac{F(z)}{z} 
\end{equation}
and 
\begin{equation}\label{M-def}
M(F,\mu,z) = L(F,\mu,z) + \frac{2F(z)}{F''(0)z^2}.
\end{equation}
The following propositions hold.

\smallskip

\begin{enumerate}
\item[(i)] We have 
\begin{align}\label{minorant-ineq}
F(x)\big\{f_{\mu}(x) - L(F,\mu,x)\big\}\ge 0
\end{align}
for all $x \in \R$ and
\begin{align}\label{minorant-interpolation}
L(F,\mu,\xi) = f_{\mu}(\xi)
\end{align}
for all $\xi \in \R$ with $F(\xi)=0$.

\smallskip

\item[(ii)] We have 
\begin{align}\label{majorant-ineq}
F(x)\big\{M(F,\mu,x) - f_{\mu}(x)\big\}\ge 0
\end{align}
for all $x \in \R$ and
\begin{align}\label{majorant-interpolation}
M(F,\mu,\xi) = f_{\mu}(\xi)
\end{align}
for all $\xi \in \R\setminus\{0\}$ with $F(\xi)=0$. At $\xi =0$ we have
\begin{align}\label{majorant-interpolation-2}
M(F,\mu,0) = 1.
\end{align}

\item[(iii)] The equality 
\begin{equation} \label{L-M-psi-eq}
\big|M(F,\mu,x) - f_{\mu}(x)\big| +  \big|f_{\mu}(x) - L(F,\mu,x)\big|=\frac{2|F(x)|}{x^2F''(0)}
\end{equation}
holds for all $x \in \R$.
\end{enumerate}
\end{proposition}

\begin{proof}
{\it Part} (i). For $x< 0$, using \eqref{Sec2_equ_conv}, \eqref{Def_A_1} and \eqref{L-def} we get
\begin{align}\label{Sec2_Prop6_eq1}
f_{\mu}(x)- L(F,\mu,x)  & = - F(x) \int_{-\infty}^0 \big\{ g'*\mu(t) - g'*\mu(0) \big\}\, e^{-tx}\, \dt,
\end{align}
and, for $x>0$, using \eqref{Sec2_equ_conv}, \eqref{Def_A_2} and \eqref{L-def} we get 
\begin{align}\label{Sec2_Prop6_eq2}
f_{\mu}(x) - L(F,\mu,x)  & =  F(x) \int_{0}^{\infty} \big\{ g'*\mu(t) - g'*\mu(0) \big\}\, e^{-tx}\, \dt.
\end{align}
If $\mc{N}(F) \geq 4$, integration by parts in \eqref{gc-trafo} shows that the Laplace transforms of $g'$ and $g''$ in the strip $0 < \re(z) < \alpha_F$ are $z/F(z)$ and $z^2/F(z)$, respectively (here we use Lemma \ref{qualitative_g_c}  (iii) to eliminate the boundary terms). Since $F(\alpha_F/2) > 0$, we conclude by Lemma \ref{qualitative_g_c}  (ii) that $g'$ and $g''$ are nonnegative on $\R$. In particular, $g'$ is also nondecreasing on $\R$. If $\mc{N}(F) = 2$ or $3$, it can be verified directly that $g'$ is nondecreasing on $\R$. In either case, this implies that $g'*\mu$ is nondecreasing, and \eqref{minorant-ineq} and \eqref{minorant-interpolation} (for $\xi \neq0$) then follow from \eqref{Sec2_Prop6_eq1} and \eqref{Sec2_Prop6_eq2}. For $\xi =0$ we see directly from \eqref{zeros-a} and \eqref{L-def} that $L(F, \mu, 0) = 0$.

\smallskip

\noindent {\it Part} (ii). For $x< 0$, using \eqref{Sec2_equ_conv}, \eqref{Def_A_1} and \eqref{M-def} we get
\begin{align}\label{Sec2_Prop6_eq3}
M(F,\mu,x) - f_{\mu}(x)  & = F(x) \int_{-\infty}^0 \left\{ g'*\mu(t) - g'*\mu(0) -\frac{2t}{F''(0)} \right\}\, e^{-tx}\, \dt,
\end{align}
and, for $x>0$, using \eqref{Sec2_equ_conv}, \eqref{Def_A_2} and \eqref{M-def} we get 
\begin{align}\label{Sec2_Prop6_eq4}
M(F,\mu,x) - f_{\mu}(x)  & = -F(x) \int_{0}^{\infty} \left\{ g'*\mu(t) - g'*\mu(0) -\frac{2t}{F''(0)} \right\}\, e^{-tx}\, \dt.
\end{align}
In order to prove \eqref{majorant-ineq} it suffices to verify that 
\begin{equation}\label{Prop6_Suff_Prove}
\big| g'*\mu(t) - g'*\mu(0)\big|\leq \frac{2 |t|}{F''(0)}
\end{equation}
for all $t \in \R$. 

\smallskip

If $\mc{N}(F) \geq 4$, we have already noted that the Laplace transform of $g''$ in the strip $0 < \re(z) < \alpha_F$ is $z^2/F(z)$. Since $F(z)/z^2$ does not vanish at the origin, we see from Lemma \ref{qualitative_g_c} that $g''(t)$ is nonnegative and decays exponentially as $|t| \to \infty$. By a direct verification, the same holds for $\mc{N}(F) =3$, where $g''$ might have one discontinuity. Thus $g''$ is integrable on $\R$ and by \eqref{gc-trafo} we find 
\begin{equation}\label{Prop6_asym}
\int_{-\infty}^{\infty} g''(t)\,\dt = 2 F''(0)^{-1}.
\end{equation}
We are now in position to prove \eqref{Prop6_Suff_Prove} for $\mc{N}(F) \geq 3$. We have already noted in part (i) that $g'$ is a nondecreasing function. Therefore, for $t >0$, we use (H2) and \eqref{Prop6_asym} to get
\begin{align*}
g'*\mu(t) - g'*\mu(0) & = \int_{-\infty}^{\infty} \big\{ g'(t - \lambda) - g'(-\lambda)\big\}\,\mu(\lambda)\,\dl\,\\
&  \leq \int_{-\infty}^{\infty} \int_0^t g''(s-\lambda)\,\ds \,\dl\\
& =2 \,t\,F''(0)^{-1}.
\end{align*}
An analogous argument holds for $t<0$. If $\mc{N}(F) =2$, we have $F(z) = \tfrac{1}{2}F''(0)\,e^{bz}\,z^2$ and $g(t) = \frac{2}{F''(0)}(t-b)\,\chi_{(b,\infty)}(t)$, and \eqref{Prop6_Suff_Prove} can be verified directly.

\smallskip

For $\xi \neq 0$, the interpolation property \eqref{majorant-interpolation} follows directly from \eqref{Sec2_Prop6_eq3} and \eqref{Sec2_Prop6_eq4}. At $\xi =0$, since $L(F, \mu, 0)=0$, it follows from \eqref{M-def} that $M(F, \mu, 0)=1$. 

\smallskip

\noindent{\it Part} (iii). Identity \eqref{L-M-psi-eq} follows easily from \eqref{L-def}, \eqref{M-def}, \eqref{minorant-ineq} and \eqref{majorant-ineq}.
\end{proof}

%%%%%%%%%%%%%%%%%%%%%%%%%%%%%%%%%%%%%%%%%%%%%%%%%%%%%%%%%%%%%%%%%%%%%%%%%%%%%%%%%%%%%%%%%%%%%%%%%%%%%%%%%%%%%%%%%%%%%%%%%%%%%%%%%%%%%%%%%%%%%%%%%%%%%%%%%%%%%%%%%%%%%%%%

\section{Proofs of the main results}

\subsection{Proof of Theorem \ref{thm1}} Recall from \eqref{Intro_H3} that under (H3) we have 

$$f_{\mu}(0^+) = \lim_{x \to 0^+}f_{\mu}(x) =1.$$

\subsubsection{Optimality} Let $L$ and $M$ be real entire functions of exponential type at most $2 \tau(E)$ such that $L(x) \leq f_{\mu}(x) \leq M(x)$ for all $x \in \R$ and 
\begin{equation*}
\int_{-\infty}^{\infty} \big\{M(x) - L(x)\big\} \, |E(x)|^{-2}\,\dx < \infty.
\end{equation*}
Since $(M-L)$ is nonnegative on $\R$, by \cite[Theorem 15]{HV} (or alternatively \cite[Lemma 14]{CL3}) we may write 
$$M(z) - L(z) = U(z)U^*(z)$$ 
with $U \in \H(E)$. Since $f_{\mu}(0^-) =0$ and $f_{\mu}(0^+) =1$, we find that $|U(0)|^2  = M(0) - L(0)  \geq 1$. From the reproducing kernel identity and the Cauchy-Schwarz inequality, it follows that
\begin{align}\label{Proof_Thm1_eq1}
1 \leq |U(0)|^2  = \big|\langle U, K(0,\cdot)\rangle_E\big|^2 \leq \|U\|_E^2\, \|K(0,\cdot)\|_E^2 =  \|U\|_E^2 \, K(0,0),
\end{align}
and therefore
\begin{equation}\label{Proof_Thm1_eq2}
\int_{-\infty}^{\infty} \big\{M(x) - L(x)\big\} \, |E(x)|^{-2}\,\dx = \int_{-\infty}^{\infty} |U(x)|^2 \, |E(x)|^{-2}\,\dx = \|U\|_E^2 \geq \frac{1}{K(0,0)}.
\end{equation}
This establishes \eqref{Intro_answer_1}. Moreover, equality in \eqref{Proof_Thm1_eq1} (and thus in \eqref{Proof_Thm1_eq2}) happens if and only if $U(z) = c\,K(0,z)$ with $|c| = K(0,0)^{-1}$. This implies that we must have 
\begin{equation}\label{Equality_cond}
M(z) - L(z) = \frac{K(0,z)^2}{K(0,0)^2}.
\end{equation}

\subsubsection{Existence} By multiplying $E$ by a complex constant of absolute value $1$, we may assume without loss of generality that $E(0) \in \R$. Since $E$ is a Hermite-Biehler function of bounded type, we see that $E^*$ also has bounded type. The companion function $B$ defined by \eqref{Intro_companion} is then a real entire function of bounded type with only real zeros. By \cite[Problem 34]{B} (see \cite{KW} for a generalization) we conclude that $B$ belongs to the Laguerre-P\'{o}lya class. The function $B$ has exponential type and it is clear that $\tau(B) \leq \tau(E)$. Note also that $B$ has a simple zero at $z=0$ (since $E(0) \neq0$ we have $K(0,0) >0$ and, by \eqref{Intro_K_z_z}, $z=0$ cannot be a double zero of $B$).

\smallskip

Applying Proposition \ref{M-L-construction} to the function $B^2(z)$, we construct the entire functions 
\begin{equation}\label{Def_L_mu}
L_{\mu}(z) = L(B^2, \mu, z)
\end{equation} 
and 
\begin{equation}\label{Def_M_mu}
M_{\mu}(z) = M(B^2, \mu, z).
\end{equation}
It follows from \eqref{minorant-ineq} and \eqref{majorant-ineq} that 
\begin{equation*}
L_{\mu}(x) \leq f_{\mu}(x) \leq M_{\mu}(x)
\end{equation*} 
for all $x \in \R$. From \eqref{al-growth}, \eqref{L-def} and \eqref{M-def} if follows that $L_{\mu}$ and $M_{\mu}$ have exponential type at most $2 \tau(E)$. Finally, from \eqref{Intro_K_w_z}, \eqref{Intro_K_z_z}, \eqref{L-def} and \eqref{M-def} we have that
\begin{align*}
M_{\mu}(z) - L_{\mu}(z) = \frac{B^2(z)}{B'(0)^2\,z^2} = \frac{K(0,z)^2}{K(0,0)^2},
\end{align*}
and as we have seen in \eqref{Equality_cond}, this is the condition for equality in \eqref{Intro_answer_1}.

\subsubsection{Uniqueness} From the equality condition \eqref{Equality_cond} and the existence of an optimal pair $\{L_{\mu},M_{\mu}\}$ we conclude that this pair must be unique.

\subsection{Proof of Theorem \ref{thm2}}

\subsubsection{Optimality} This follows as in the optimality part of Theorem \ref{thm1}, just observing that 
$$\wt{f}_{\mu}(0^-) =-1  \ \ \ {\rm and} \ \ \  \wt{f}_{\mu}(0^+) =1.$$

\subsubsection{Existence} We use Proposition \ref{M-L-construction} with the Laguerre-P\'{o}lya functions $B^2(z)$ and its reflection $B^2(-z)$ to define 
\begin{equation}\label{Def_L_mu_t}
\wt{L}_{\mu}(z) = L(B^2(z), \mu, z) - M(B^2(-z), \mu, -z)
\end{equation}
and
\begin{equation}\label{Def_M_mu_t}
\wt{M}_{\mu}(z) = M(B^2(z), \mu, z) - L(B^2(-z), \mu, -z).
\end{equation}
These are real entire functions of exponential type at most $2\tau(E)$ that satisfy
\begin{equation*}
\wt{L}_{\mu}(x) \leq \wt{f}_{\mu}(x) \leq \wt{M}_{\mu}(x)
\end{equation*} 
for all $x \in \R$. As before, from \eqref{Intro_K_w_z}, \eqref{Intro_K_z_z}, \eqref{L-def} and \eqref{M-def} we find that
\begin{align*}
\wt{M}_{\mu}(z) - \wt{L}_{\mu}(z) = \frac{2B^2(z)}{B'(0)^2\,z^2} = \frac{2K(0,z)^2}{K(0,0)^2},
\end{align*}
and this is the condition for equality in \eqref{Intro_answer_1_o}.

\subsubsection{Uniqueness} It follows as in the proof of Theorem \ref{thm1}.

\subsection{Further results} Without assuming (H3) it is possible to solve the minorant problem for $f_{\mu}$.  However, we do have to assume that the companion function that generates the nodes of interpolation does not belong to the space $\H(E)$.

\begin{corollary}\label{Cor7}
Let $E$ be a Hermite-Biehler function of bounded type in $\U$ such that $E(0) > 0$. Let $\mu$ be a signed Borel measure on $\R$ satisfying  {\rm (H1) - (H2)}. Assume that $\supp(\mu) \subset [-2\tau(E), \infty)$ and let $f_{\mu}$ be defined by \eqref{Intro_Def_f_mu}. Assume that $B \notin \H(E)$. Let $L_{\mu}$ be the real entire function of exponential type at most $2\tau(E)$ defined by \eqref{Def_L_mu}. If $L:\C \to \C$ is a real entire function of exponential type at most $2 \tau(E)$ such that
\begin{equation*}
L(x) \leq f_{\mu}(x)
\end{equation*}
for all $x \in \R$, then
\begin{equation}\label{Add_res_1}
\int_{-\infty}^{\infty} \big\{f_{\mu}(x) - L(x)\big\}\, |E(x)|^{-2}\,\dx \geq \int_{-\infty}^{\infty} \big\{f_{\mu}(x) - L_{\mu}(x)\big\}\, |E(x)|^{-2}\,\dx.
\end{equation} 
\end{corollary}

\begin{proof}
From \eqref{L-M-psi-eq} and \eqref{Def_L_mu} we observe first that the right-hand side of \eqref{Add_res_1} is indeed finite. If the left-hand side of \eqref{Add_res_1} is $+\infty$ there is nothing to prove. Assume then that $(f_\mu - L) \in L^1(\R, |E(x)|^{-2}\,\dx)$. We use the fact that there exists a majorant $M_{\mu}$ (not necessarily extremal anymore) defined by \eqref{Def_M_mu}, and from \eqref{L-M-psi-eq} we see that $(M_{\mu} - f_\mu) \in L^1(\R, |E(x)|^{-2}\,\dx)$. By the triangle inequality we get 
$(M_{\mu} - L_\mu) \in L^1(\R, |E(x)|^{-2}\,\dx)$ and $(M_{\mu} - L) \in L^1(\R, |E(x)|^{-2}\,\dx)$. Since the last two functions are nonnegative on $\R$, from \cite[Theorem 15]{HV} (or alternatively \cite[Lemma 14]{CL3}) we can write 
$$M_{\mu}(z) - L(z) = U(z)U^*(z)$$
and
$$M_{\mu}(z) - L_{\mu}(z) = V(z)V^*(z),$$
with $U,V \in \H(E)$. This gives us 
$$L_{\mu}(z) - L(z) = U(z)U^*(z) - V(z)V^*(z).$$
Since $B \notin \H(E)$, from \cite[Theorem 22]{B} the set $\{z \mapsto \ov{E}(\xi)^{-1}K(\xi, z);\, B(\xi) = 0\}$ is an orthogonal basis for $\H(E)$ (note here that if $E(\xi) = 0$, the function  $\ov{E}(\xi)^{-1}K(\xi, z)$ has to be interpreted as the appropriate limit). We now use Parseval's identity and the the fact that $L_{\mu}$ interpolates $f_{\mu}$ at the zeros of $B$ to get
\begin{align*}
\int_{-\infty}^{\infty}& \big\{L_{\mu}(x)  - L(x)\big\}\, |E(x)|^{-2}\,\dx  = \int_{-\infty}^{\infty} \big\{|U(x)|^2 - |V(x)|^2\big\}\, |E(x)|^{-2}\,\dx \\
& = \sum_{B(\xi) = 0} \frac{ \big\{|U(\xi)|^2 - |V(\xi)|^2\big\}}{K(\xi,\xi)} =  \sum_{B(\xi) = 0} \frac{ \big\{L_{\mu}(\xi) - L(\xi)\big\}}{K(\xi,\xi)} =  \sum_{B(\xi) = 0} \frac{ \big\{f_{\mu}(\xi) - L(\xi)\big\}}{K(\xi,\xi)}\\
& \geq 0.
\end{align*}
This concludes the proof of the corollary.

\end{proof}

When $f_{\mu} \in L^1(\R, |E(x)|^{-2}\,\dx)$ it is possible to determine the precise values of the optimal integrals in our extremal problem separately. 

\begin{corollary}\label{Prop7_sep}
Let $E$ be a Hermite-Biehler function of bounded type in $\U$ such that $E(0) > 0$.  Let $\mu$ be a signed Borel measure on $\R$ satisfying  {\rm (H1) - (H2)}. Assume that $\supp(\mu) \subset [-2\tau(E), \infty)$ and let $f_{\mu}$ be defined by \eqref{Intro_Def_f_mu}. Assume that 
\begin{equation}\label{Prop7_L1_cond}
\int_{-\infty}^{\infty} |f_{\mu}(x)|\, |E(x)|^{-2}\,\dx <\infty
\end{equation}
and that $B \notin \H(E)$. 
\begin{enumerate}
\item[(i)] Let $L_{\mu}$ be the extremal minorant of exponential type at most $2\tau(E)$ defined by \eqref{Def_L_mu}. We have
\begin{equation}\label{Prop7_eq1_L}
\int_{-\infty}^{\infty} L_{\mu}(x)\, |E(x)|^{-2}\,\dx = \sum_{\stackrel{\xi > 0}{B(\xi)=0}} \frac{f_{\mu}(\xi)}{K(\xi, \xi)}.
\end{equation}
\item[(ii)] Assuming {\rm(H3)}, let $M_{\mu}$ be the extremal majorant of exponential type at most $2\tau(E)$ defined by \eqref{Def_M_mu}. We have
\begin{equation}\label{Prop7_eq2_M}
\int_{-\infty}^{\infty} M_{\mu}(x)\, |E(x)|^{-2}\,\dx = \frac{1}{K(0,0)} + \sum_{\stackrel{\xi > 0}{B(\xi)=0}} \frac{f_{\mu}(\xi)}{K(\xi, \xi)}.
\end{equation}
\end{enumerate}
\end{corollary}

\begin{proof} We first prove (ii). The function $M_{\mu}$ is nonnegative on $\R$ and belongs to $L^1(\R, |E(x)|^{-2}\,\dx)$ from \eqref{Prop7_L1_cond} (observe in particular that $E$ cannot have nonnegative zeros in this situation). From \cite[Theorem 15]{HV} (or alternatively \cite[Lemma 14]{CL3}) we can write 
\begin{equation}\label{Dec_M_U_Us}
M_{\mu}(z) = U(z)U^*(z)
\end{equation} 
with $U \in \H(E)$. We use again the fact that the set $\{z \mapsto  \ov{E}(\xi)^{-1}K(\xi, z);\, B(\xi) = 0\}$ is an orthogonal basis for $\H(E)$ since $B \notin \H(E)$ \cite[Theorem 22]{B}. From Parseval's identity and the the fact that $M_{\mu}$ interpolates $f_{\mu}$ at the zeros of $B$ (with $M_{\mu}(0) =1$) we arrive at
\begin{align*}
\int_{-\infty}^{\infty} M_{\mu}(x)\, |E(x)|^{-2}\,\dx & = \int_{-\infty}^{\infty} |U(x)|^2\, |E(x)|^{-2}\,\dx= \sum_{B(\xi) = 0} \frac{|U(\xi)|^2}{K(\xi,\xi)} = \sum_{B(\xi) = 0} \frac{M_{\mu}(\xi)}{K(\xi,\xi)} \\
& =  \frac{1}{K(0,0)} + \sum_{\stackrel{\xi > 0}{B(\xi)=0}} \frac{f_{\mu}(\xi)}{K(\xi, \xi)}.
\end{align*}
This establishes \eqref{Prop7_eq2_M}. 

\smallskip

We now prove (i). In this case, we still have a majorant $M_{\mu}$ (not necessarily extremal anymore) and the factorization \eqref{Dec_M_U_Us} still holds. From \eqref{L-M-psi-eq} we see that $(M_{\mu} - L_\mu) \in L^1(\R, |E(x)|^{-2}\,\dx)$ and we can write again $M_{\mu}(z) - L_{\mu}(z) = V(z) V^*(z)$, with $V \in \H(E)$. This gives us
\begin{equation}\label{Dec_L_U_Us}
L_{\mu}(z) = U(z)U^*(z) - V(z) V^*(z).
\end{equation}
Using Parseval's identity again, and the fact that $L_{\mu}$ interpolates $f_{\mu}$ at the zeros of $B$, we arrive at 
\begin{align*}
\int_{-\infty}^{\infty} L_{\mu}(x)\, |E(x)|^{-2}\,\dx & = \int_{-\infty}^{\infty}\big\{ |U(x)|^2 - |V(x)|^2\big\}\, |E(x)|^{-2}\,\dx= \sum_{B(\xi) = 0} \frac{|U(\xi)|^2 - |V(\xi)|^2}{K(\xi,\xi)} \\
& = \sum_{B(\xi) = 0} \frac{L_{\mu}(\xi)}{K(\xi,\xi)}  =  \sum_{\stackrel{\xi > 0}{B(\xi)=0}} \frac{f_{\mu}(\xi)}{K(\xi, \xi)}.
\end{align*}
This establishes \eqref{Prop7_eq1_L} and completes the proof.
\end{proof}

\begin{corollary}\label{Prop8_sep}
Let $E$ be a Hermite-Biehler function of bounded type in $\U$ such that $E(0) > 0$. Let $\mu$ be a signed Borel measure on $\R$ satisfying  {\rm (H1) - (H2) - (H3)}. Assume that $\supp(\mu) \subset [-2\tau(E), \infty)$ and let $\wt{f}_{\mu}$ be defined by \eqref{Intro_Def_f_mu_o}. Assume that 
\begin{equation}\label{Prop8_L1_cond}
\int_{-\infty}^{\infty} |\wt{f}_{\mu}(x)|\, |E(x)|^{-2}\,\dx <\infty
\end{equation}
and that $B \notin \H(E)$. Let $\wt{L}_{\mu}$ and $\wt{M}_{\mu}$ be the extremal functions of exponential type at most $2\tau(E)$ defined by \eqref{Def_L_mu_t} and \eqref{Def_M_mu_t}, respectively. We have
\begin{equation*}
\int_{-\infty}^{\infty} \wt{L}_{\mu}(x)\, |E(x)|^{-2}\,\dx = -\frac{1}{K(0,0)} + \sum_{\stackrel{\xi \neq 0}{B(\xi)=0}} \frac{\wt{f}_{\mu}(\xi)}{K(\xi, \xi)}
\end{equation*}
and 
\begin{equation*}
\int_{-\infty}^{\infty} \wt{M}_{\mu}(x)\, |E(x)|^{-2}\,\dx = \frac{1}{K(0,0)} + \sum_{\stackrel{\xi \neq 0}{B(\xi)=0}} \frac{\wt{f}_{\mu}(\xi)}{K(\xi, \xi)}.
\end{equation*}
\end{corollary}

\begin{proof} From the integrability condition \eqref{Prop8_L1_cond} we see that $E$ cannot have real zeros and we may use \eqref{Def_L_mu_t}, \eqref{Def_M_mu_t}, \eqref{Dec_M_U_Us} and \eqref{Dec_L_U_Us} to write
\begin{equation*}
\wt{L}_{\mu}(z) = \big(U_1(z)U_1^*(z) - V_1(z)V_1^*(z)\big) - U_2(z)U_2^*(z)
\end{equation*}
and
\begin{equation*}
\wt{M}_{\mu}(z) = U_3(z)U_3^*(z) - \big(U_4(z)U_4^*(z) - V_4(z)V_4^*(z) \big),
\end{equation*}
where $U_i, V_j \in \H(E)$. Once we have completed this passage from $L^1$ to $L^2$, the remaining steps are analogous to the proof of Corollary \ref{Prop7_sep}.
\end{proof}

\section{Periodic analogues}\label{Per_Analogues}
Recall that we write $e(z) = e^{2\pi i z}$ for $z \in \C$. In this section we consider the problem of one-sided approximation of periodic functions by trigonometric polynomials of a given degree, as described in \S \ref{par_Per}. The main tools we use here are the theory of reproducing kernel Hilbert spaces of polynomials and the theory of orthogonal polynomials in the unit circle, and we start by reviewing the terminology and the basic facts of these two well-established subjects. In doing so, we follow the notation of \cite{CL3, Li, S} to facilitate some of the references. 

\subsection{Preliminaries}
\subsubsection{Reproducing kernel Hilbert spaces of polynomials} \label{RKHSP}
We write $\ud = \{z \in \C; |z| <1\}$ for the open unit disc and $\uc$ for the unit circle. Let $n\in\Z^{+}$ and let $\mathcal{P}_n$ be the set of polynomials of degree at most $n$ with complex coefficients. If $Q \in \mc{P}_n$ we define the conjugate polynomial $Q^{*,n}$ by
\begin{equation}\label{antiunitary_map}
Q^{*,n}(z) = z^n \, \overline{Q\big(\bar{z}\,^{-1}\big)}.
\end{equation}
If $Q$ has exact degree $n$, we sometimes omit the superscript $n$ and write $Q^*$ for simplicity.

\smallskip

Let $P$ be a polynomial of exact degree $n+1$ with no zeros on $\uc$ such that 
\begin{align}\label{P-ineq}
|P^*(z)|<|P(z)|
\end{align}
for all $z\in\ud$. We consider the Hilbert space $\dbpn$ consisting of the elements in $\mathcal{P}_n$ with scalar product
\begin{equation}\label{ET_inner_prod}
\langle Q,R \rangle_{\dbpn} = \int_{\R/\Z} Q(e(x)) \,\overline{R(e(x))} \, |P(e(x))|^{-2}\,\dx.
\end{equation}
From Cauchy's integral formula, it follows easily that the reproducing kernel for this finite-dimensional Hilbert space is given by
$$\ms{K}(w,z) = \frac{P(z)\overline{P(w)} - P^*(z)\overline{P^*(w)}}{1-\bar{w}z}\,,$$
i.e. for every $w \in \C$ we have the identity 
$$\langle Q, \ms{K}(w,\cdot)\rangle_{\dbpn} = Q(w).$$
As before, we define the companion polynomials
\begin{equation}\label{Intro_companion_ET}
\ms{A}(z) := \frac12 \big\{P(z) + P^*(z)\big\} \ \ \ {\rm and}  \ \ \ \ms{B}(z) := \frac{i}{2}\big\{P(z) - P^*(z)\big\}\,,
\end{equation}
and we find that $\ms{A} = \ms{A}^*$,  $\ms{B}= \ms{B}^*$ and $P(z) = \ms{A}(z) - i\ms{B}(z)$. Since the coefficients of $z^0$ and $z^{n+1}$ of $P$ do not have the same absolute value (this would contradict \eqref{P-ineq} at $z=0$) the polynomials $\ms{A}$ and $\ms{B}$ have exact degree $n+1$. From \eqref{P-ineq} we also see that $\ms{A}$ and $\ms{B}$ have all of their zeros in $\uc$. 

\smallskip

The reproducing kernel has the alternative representation
\begin{equation}\label{Sec8_Rep2}
\ms{K}(w,z) = \frac2i \left(\frac{\ms{B}(z) \overline{\ms{A}(w)} - \ms{A}(z)\overline{\ms{B}(w)} }{ 1-\bar{w} z}\right).
\end{equation}
Observe that 
$$\ms{K}(w,w) = \langle \ms{K}(w,\cdot),\ms{K}(w,\cdot)\rangle_{\dbpn} \geq 0$$
for all $w \in \C$. If there is $w \in \C$ such that $\ms{K}(w,w)=0$, then $\ms{K}(w,\cdot) \equiv0$ and $Q(w) = 0$ for every $Q\in\mathcal{P}_{n}$, a contradiction. Therefore $\ms{K}(w,w)>0$ for all $w\in\C$. From the representation \eqref{Sec8_Rep2} it follows that $\ms{A}$ and $\ms{B}$ have only simple zeros and their zeros never agree. 

\smallskip

From \eqref{Sec8_Rep2} we see that the sets $\{z \mapsto \ms{K}(\zeta,z);\  \ms{A}(\zeta) =0\}$ and $\{z \mapsto \ms{K}(\zeta,z);\  \ms{B}(\zeta) =0\}$ are orthogonal bases for $\dbpn$ and, in particular, we arrive at Parseval's formula (see \cite[Theorem 2]{Li}) 
\begin{align}\label{poly-parseval}
||Q||^2_{\dbpn}= \sum_{\ms{A}(\zeta) =0} \frac{|Q(\zeta)|^2}{\ms{K}(\zeta,\zeta)} = \sum_{\ms{B}(\zeta) =0} \frac{|Q(\zeta)|^2}{\ms{K}(\zeta,\zeta)}.
\end{align}

\subsubsection{Orthogonal polynomials in the unit circle} The map $x \mapsto e(x)$ allows us to identify measures on $\R/\Z$ with measures on the unit circle $\uc$. Let $\vartheta$ be a nontrivial probability measure on $\R/\Z \sim \uc$ (recall that $\vartheta$ is trivial if it has support on a finite number of points) and consider the space $L^2(\uc, \dvar)$ with inner product given by 
\begin{equation*}
\langle f,g\rangle_{L^2(\uc, \dvar)} = \int_{\uc} f(z)\,\overline{g(z)}\, \dvar(z) = \int_{\R/\Z} f(e(x))\,\overline{g(e(x))}\, \dvar(x). 
\end{equation*}
We define the {\it monic orthogonal polynomials} $\Phi_n(z) = \Phi_n(z;\dvar)$ by the conditions 
$$\Phi_n(z) = z^n + \text{lower order terms}\,;\ \ \ \ \ \  \langle \Phi_n, z^j\rangle_{L^2(\uc, \dvar)} =0\qquad (0\le j<n);$$
and we define the {\it orthonormal polynomials} by $\varphi_n = c_n\Phi_n/||\Phi_n||_2$, where $c_n$ is a complex number of absolute value one such that $\varphi_n(1) \in \R$ (this normalization will be used later). Observe that 
\begin{equation}\label{antiunitary_2}
\langle Q^{*,n},R^{*,n}\rangle_{L^2(\uc, \dvar)}  = \langle R,Q\rangle_{L^2(\uc, \dvar)} 
\end{equation} 
for all polynomials $Q,R \in \mc{P}_n$, where the conjugation map $*$ was defined in \eqref{antiunitary_map}. The next lemma collects the relevant facts for our purposes from B. Simon's survey article \cite{S}.

\begin{lemma} \label{Sec8_Lem25} 
Let $\vartheta$ be a nontrivial probability measure on $\R/\Z$.
\begin{enumerate}
\item[(i)] $\varphi_n$ has all its zeros in $\ud$ and $\varphi_n^*$ has all its zeros in $\C\backslash\overline{\ud}$. 
\smallskip
\item[(ii)] Define a new measure $\vartheta_n$ on $\R/\Z$ by
$$\dvar_n (x)= \frac{\dx}{\big|\varphi_n(e(x); \d\vartheta)\big|^2},$$
Then $\vartheta_n$ is a probability measure on $\R/\Z$, $\varphi_j(z;\d\vartheta) =  \varphi_j(z;\d\vartheta_n)$ for $j=0,1,\ldots,n$ and for all $Q,R \in \mathcal{P}_{n}$ we have
\begin{equation}\label{Sec8_equ_measures}
\langle Q, R \rangle_{L^2(\uc,\d\vartheta)} = \langle Q, R \rangle_{L^2(\uc,\d\vartheta_n)}.\end{equation}
\end{enumerate}
\end{lemma}

\begin{proof} (i) This is \cite[Theorem 4.1]{S}.

\smallskip

\noindent (ii) This follows from \cite[Theorem 2.4, Proposition 4.2 and Theorem 4.3]{S}.
\end{proof}

Let $n \geq 0 $ and $\varphi_{n+1}(z)  =  \varphi_{n+1}(z;\d\vartheta)$. By Lemma \ref{Sec8_Lem25} (i) and the maximum principle we have
\begin{align*}\label{HB-analogue}
|\varphi_{n+1}(z)| < |\varphi_{n+1}^*(z)|
\end{align*}
for all $z\in \ud$. By Lemma \ref{Sec8_Lem25} (ii) we note (Christoffel-Darboux formula) that $\mathcal{P}_n$ with the scalar product $\langle \cdot,\cdot\rangle_{L^2(\uc,\d\vartheta)}$ is a reproducing kernel Hilbert space with reproducing kernel given by 
\begin{equation}\label{Sec8_defKn}
\ms{K}_{n}(w,z) =\frac{\varphi^*_{n+1}(z)\, \overline{\varphi_{n+1}^*(w)} - \varphi_{n+1}(z)\,\overline{\varphi_{n+1}(w)}}{1-\bar{w}z}.
\end{equation}
Observe that $\varphi^*_{n+1}$ plays the role of $P$ in \S \ref{RKHSP}. As before, we define the two companion polynomials (here we use the subscript according to the degree of the polynomial)
\begin{equation}\label{Sec8_DefAn}
\ms{A}_{n+1}(z) = \frac{1}{2}\big\{\varphi_{n+1}^*(z) + \varphi_{n+1}(z)\big\} \ \ \ \ {\rm and} \ \ \ \ \ms{B}_{n+1}(z) = \frac{i}{2}\big\{ \varphi_{n+1}^*(z) - \varphi_{n+1}(z)\big\},
\end{equation}
and we note that \eqref{poly-parseval} holds.

\smallskip

We now derive the quadrature formula that is suitable for our purposes. This result appears in \cite[Corollary 26]{CL3} and we present a short proof here for convenience.

\begin{proposition}\label{Sec8_Cor26}
Let $\vartheta$ be a nontrivial probability measure on $\R/\Z$ and let $\ms{W}:\C \to \C$ be a trigonometric polynomial of degree at most $N$. Let $\varphi_{N+1}(z) = \varphi_{N+1}(z;\d\vartheta)$ be the $(N+1)$-th orthonormal polynomial in the unit circle with respect to this measure and consider $\ms{K}_N$, $\ms{A}_{N+1}$ and $\ms{B}_{N+1}$ as defined in \eqref{Sec8_defKn} and \eqref{Sec8_DefAn}. Then we have 
\begin{align*}
\int_{\R/\Z} \ms{W}(x)\,\d\vartheta(x)   = \sum_{\stackrel{\xi \in \R/\Z}{\ms{A}_{N+1}(e(\xi))=0}} \frac{\ms{W}(\xi)}{\ms{K}_{N}(e(\xi), e(\xi))} = \ \sum_{\stackrel{\xi \in \R/\Z}{\ms{B}_{N+1}(e(\xi))=0}}\frac{\ms{W}(\xi)}{\ms{K}_{N}(e(\xi), e(\xi))}. 
\end{align*}
\end{proposition}

\begin{proof} Write
$$\ms{W}(z) = \sum_{k=-N}^{N} a_k \,e(kz)$$
and assume first that $\ms{W}$ is real valued on $\R$, i.e. $a_k = \overline{a_{-k}}$. Let $\tau = \min_{x\in\R} \ms{W}(x).$
Then $z \mapsto \ms{W}(z) - \tau$ is a real trigonometric polynomial of degree at most $N$ that is nonnegative on $\R$. By the Riesz-F\'{e}jer theorem there exists a polynomial $Q \in \mc{P}_N$ such that
$$\ms{W}(z) - \tau = Q(e(z))\,\overline{Q(e(\ov{z}))}$$
for all $z \in \C$. Writing $\tau= |\tau_1|^2 - |\tau_2|^2$, and using \eqref{Sec8_equ_measures} and \eqref{poly-parseval}, we obtain 
\begin{align*}
\int_{\R/\Z} \ms{W}(x)\,\d\vartheta(x) &= \int_{\R/\Z} \big\{|Q(e(x))|^2 +|\tau_1|^2 - |\tau_2|^2\big\} \,\d\vartheta(x) \\
& =  \int_{\R/\Z} \frac{|Q(e(x))|^2 +|\tau_1|^2 - |\tau_2|^2}{\big|\varphi_{N+1}(e(x))\big|^2}\ \dx \\
& = \sum_{\stackrel{\xi \in \R/\Z}{\ms{B}_{N+1}(e(\xi))=0}}  \frac{|Q(e(\xi))|^2 +|\tau_1|^2 - |\tau_2|^2}{\ms{K}_{N}(e(\xi), e(\xi))}\\
& =\sum_{\stackrel{\xi \in \R/\Z}{\ms{B}_{N+1}(e(\xi))=0}} \frac{\ms{W}(\xi)}{\ms{K}_{N}(e(\xi), e(\xi))},
\end{align*}
and analogously at the nodes given by the roots of $\ms{A}_{N+1}$. The general case follows by writing $\ms{W}(z) = \ms{W}_1(z)- i\ms{W}_2(z)$, with $\ms{W}_1(z) =  \sum_{k=-N}^N b_k\,e(kz)$ and $\ms{W}_2(z) =  \sum_{k=-N}^N c_k \,e(kz)$, where $b_k = \frac12(a_k + \overline{a_{-k}})$ and $c_k =  \frac i 2 (a_k - \overline{a_{-k}})$. 
\end{proof}

\subsection{Extremal trigonometric polynomials} We now present the solution of the extremal problem \eqref{Trig_EP1} - \eqref{Trig_EP2} for a class of periodic functions with a certain exponential subordination. As described below, this class comes from the periodization of the functions $f_{\mu}$ and $\wt{f}_{\mu}$ defined in \eqref{Intro_Def_f_mu} and \eqref{Intro_Def_f_mu_o}.

\subsubsection{Defining the periodic analogues} Throughout this section we let $\mu$ be a (locally finite) signed Borel measure on $\R$ satisfying conditions (H1') - (H2). The condition (H1') is simply a restriction of our current (H1), namely: 

\smallskip

\begin{enumerate}

\item[(H1')] The measure $\mu$ has support on $[0,\infty)$.

\end{enumerate}

\smallskip

\noindent When convenient, we may require additional properties on $\mu$. The first one is our usual (H3), and we now introduce the following summability condition:

\smallskip

\begin{enumerate}

\item[(H4)] The distribution function $\mu (x):= \mu((-\infty, x])$ verifies

\begin{equation*}
\int_{0}^{\infty} \frac{1}{\lambda^2}\,\mu(\lambda)\,\dl < \infty.
\end{equation*}
\end{enumerate}

\medskip

\noindent For $\lambda >0$ we consider the following truncated function that appears on the right-hand side of \eqref{Intro_int_parts}:

\begin{equation*}
v(\lambda, x) = \left\{
\begin{array}{cc}
xe^{-\lambda x} & {\rm if} \ x >0;\\
0, & {\rm if} \ x \leq 0,\\
\end{array}
\right.
\end{equation*}
and define the $1$-periodic function 
\begin{equation*}
h(\lambda,x) := \sum_{n \in \Z} v(\lambda, x+n) =  \frac{e^{-\lambda(x - \lfloor x \rfloor - \frac12)}\big\{2\sinh(\lambda/2)(x - \lfloor x \rfloor - \frac12) +\cosh(\lambda/2) \big\}}{4\sinh(\lambda/2)^2}.\end{equation*}
If $\mu$ is a signed Borel measure satisfying (H1') - (H2) - (H4) we define the $1$-periodic function 
\begin{align}\label{Def_ET_F_mu}
\ms{F}_{\mu}(x) := \int_{0}^{\infty} h(\lambda,x) \,\mu(\lambda)\,\dl = \sum_{n \in \Z} f_{\mu}(x + n),
\end{align}
where the last equality follows from \eqref{Intro_int_parts} and Fubini's theorem. We observe that $\ms{F}_{\mu}$ is differentiable for $x \notin \Z$ and that 
$$ \ms{F}_{\mu}(0^-) = \ms{F}_{\mu}(0).$$ 
For $0 \leq x \leq 1$ we have
\begin{align*}
h(\lambda,x) &= xe^{-\lambda x} + \frac{e^{-\lambda}}{\big(1 - e^{-\lambda}\big)^2}\, \Big( x e^{-\lambda x} \big(1 - e^{-\lambda}\big) + e^{-\lambda x}\Big),
\end{align*}
and we see from dominated convergence and the computation in \eqref{Intro_H3} that 
\begin{equation*}
\limsup_{x \to 0^+} \ms{F}_{\mu}(x) \leq \ms{F}_{\mu}(0) + 1,
\end{equation*}
and {\it under the additional condition} (H3) we have
\begin{equation}\label{ET_Sec4_H3_cond}
\ms{F}_{\mu}(0^+) = \ms{F}_{\mu}(0^-) + 1 = \ms{F}_{\mu}(0) + 1.
\end{equation}

We now define the odd counterpart. First we let, for $\lambda >0$, 
\begin{equation*}
\wt{v}(\lambda, x) := v(\lambda, x) - v(\lambda, -x)
\end{equation*}
and consider the $1$-periodic function
\begin{align*}
\wt{h}(\lambda,x) & := \sum_{n \in \Z} \wt{v}(\lambda, x+n) \\
& = \frac{- \tfrac12 \cosh (\lambda /2) \sinh\big( \lambda(x - \lfloor x \rfloor - \tfrac12)\big) +( x - \lfloor x \rfloor - \tfrac12)  \sinh(\lambda/2)\cosh\big( \lambda(x - \lfloor x \rfloor - \tfrac12)\big)}{\sinh(\lambda/2)^2}. 
\end{align*}
If $\mu$ is a signed Borel measure satisfying (H1') - (H2) we define the odd $1$-periodic function
\begin{align}\label{Def_ET_F_tilde_mu}
\wt{\ms{F}}_{\mu}(x) := \int_{0}^{\infty} \wt{h}(\lambda,x) \,\mu(\lambda)\,\dl.
\end{align}
Note that {\it we do not have to assume} (H4) in order to define $\wt{\ms{F}}_{\mu}$ in \eqref{Def_ET_F_tilde_mu} since, for all $x\in \R$, the function $\lambda \mapsto \wt{h}(\lambda,x)$ is $O(\lambda)$ as $\lambda \to 0$. If, however, we have (H4), the function $\ms{F}_{\mu}$ is well-defined and we have
\begin{align*}
\wt{\ms{F}}_{\mu}(x) = \ms{F}_{\mu}(x) - \ms{F}_{\mu}(-x) = \sum_{n \in \Z} \wt{f}_{\mu}(x + n),
\end{align*}
verifying that $\wt{\ms{F}}_{\mu}$ is in fact the periodization of $\wt{f}_{\mu}$. We note that $\wt{\ms{F}}_{\mu}$ is differentiable for $x \notin \Z$. For $0\leq  x\leq 1$ we may write alternatively
\begin{align}\label{Estimate_h_tilde}
\begin{split}
\wt{h}(\lambda,x) & = h(\lambda,x) - h(\lambda,1-x)\\
& = xe^{-\lambda x} - (1-x)e^{-\lambda (1-x)} \\
&  \ \ \ \ \ \ \ \ \ \ \ + \frac{e^{-\lambda}}{\big(1 - e^{-\lambda}\big)^2}\, \Big( x e^{-\lambda x} \big(1 - e^{-\lambda}\big) + e^{-\lambda x} - (1-x) e^{-\lambda (1-x)} \big(1 - e^{-\lambda}\big) -  e^{-\lambda (1-x)} \Big),
\end{split}
\end{align}
and we may use dominated convergence in \eqref{Def_ET_F_tilde_mu} together with the computation in \eqref{Intro_H3} to conclude that, under (H1') - (H2) - (H3), we have
\begin{equation*}
\wt{\ms{F}}_{\mu}(0^\pm) = \pm 1. 
\end{equation*}

We highlight the fact that when $\mu$ is the Dirac delta measure, we recover the sawtooth function (multiplied by $-2$) in \eqref{Def_ET_F_tilde_mu}. In fact, observing that for $x \notin \Z$ we have
$$\wt{h}(\lambda,x)  = - \frac{\partial}{\partial \lambda} \left(\frac{\sinh\big(-\lambda(x - \lfloor x \rfloor - \frac12)\big)}{\sinh(\lambda/2)}\right),$$
we find, for $x \notin \Z$, 
$$\wt{\ms{F}}_{\mu}(x) = \int_{0}^{\infty} \wt{h}(\lambda,x) \,\dl = -2(x - \lfloor x \rfloor - \tfrac12).$$
This is expected since the corresponding $\wt{f}_{\mu}$ is the signum function. In particular, the results we present below extend the work of Li and Vaaler \cite{LV} on the sawtooth function.

\subsubsection{Main results} The following two results provide a complete solution of the extremal problem \eqref{Trig_EP1} - \eqref{Trig_EP2} for the periodic functions $\ms{F}_{\mu}$ and $\wt{\ms{F}}_{\mu}$ defined in \eqref{Def_ET_F_mu} and \eqref{Def_ET_F_tilde_mu}, with respect to arbitrary nontrivial probability measures $\vartheta$. This completes the framework initiated in \cite{CL3}, where this extremal problem was solved for an analogous class of even periodic functions with exponential subordination. In what follows we let $\varphi_{N+1}(z) = \varphi_{N+1}(z;\d\vartheta)$ be the $(N+1)$-th orthonormal polynomial in the unit circle with respect to this measure and consider $\ms{K}_N, \ms{A}_{N+1}, \ms{B}_{N+1}$ as defined in \eqref{Sec8_defKn} and \eqref{Sec8_DefAn}. 

\begin{theorem} \label{thm11_ET}
Let $\mu$ be a signed Borel measure on $\R$ satisfying {\rm (H1') - (H2) - (H4)}, and let $\ms{F}_{\mu}$ be defined by \eqref{Def_ET_F_mu}. Let $\vartheta$ be a nontrivial probability measure on $\R/\Z$ and $N \in \Z^+$. 

\begin{enumerate}
\item[(i)] If $\ms{L}:\C \to \C$ is a real trigonometric polynomial of degree at most $N$ such that
\begin{equation}\label{Intro_eq_L_ET}
\ms{L}(x) \leq \ms{F}_{\mu}(x) 
\end{equation}
for all $x \in \R/\Z$, then 
\begin{equation}\label{Intro_answer_L_ET}
\int_{\R/\Z}  \ms{L}(x)\, \dvar(x) \leq \frac{\ms{F}_{\mu}(0)}{\ms{K}_N(1,1)}\ +  \sum_{\stackrel{\xi \in \R/\Z\,;\, \xi \neq 0}{\ms{B}_{N+1}(e(\xi))=0}}\frac{\ms{F}_{\mu}(\xi)}{\ms{K}_{N}(e(\xi), e(\xi))}.
\end{equation} 
Moreover, there is a unique real trigonometric polynomial $\ms{L}_{\mu}:\C \to \C$ of degree at most $N$ satisfying \eqref{Intro_eq_L_ET} for which the equality in \eqref{Intro_answer_L_ET} holds.

\smallskip

\item[(ii)] Assume that $\mu$ also satisfies {\rm (H3)}. If $\ms{M}:\C \to \C$ is a real trigonometric polynomial of degree at most $N$ such that
\begin{equation}\label{Intro_eq_M_ET}
\ms{F}_{\mu}(x) \leq \ms{M}(x)  
\end{equation}
for all $x \in \R/\Z$, then 
\begin{equation}\label{Intro_answer_M_ET}
\int_{\R/\Z}  \ms{M}(x)\, \dvar(x) \geq \frac{\ms{F}_{\mu}(0^+)}{\ms{K}_N(1,1)}\ +  \sum_{\stackrel{\xi \in \R/\Z\,;\, \xi \neq 0}{\ms{B}_{N+1}(e(\xi))=0}}\frac{\ms{F}_{\mu}(\xi)}{\ms{K}_{N}(e(\xi), e(\xi))}.
\end{equation} 
Moreover, there is a unique real trigonometric polynomial $\ms{M}_{\mu}:\C \to \C$ of degree at most $N$ satisfying \eqref{Intro_eq_M_ET} for which the equality in \eqref{Intro_answer_M_ET} holds.
\end{enumerate}
\end{theorem}

\begin{theorem}  \label{thm12_ET}
Let $\mu$ be a signed Borel measure on $\R$ satisfying {\rm (H1') - (H2) - (H3)}, and let $\wt{\ms{F}}_{\mu}$ be defined by \eqref{Def_ET_F_tilde_mu}. Let $\vartheta$ be a nontrivial probability measure on $\R/\Z$ and $N \in \Z^+$. 
\begin{enumerate}
\item[(i)] If $\ms{L}:\C \to \C$ is a real trigonometric polynomial of degree at most $N$ such that
\begin{equation}\label{Intro_eq_L_ET_case2}
\ms{L}(x) \leq \wt{\ms{F}}_{\mu}(x) 
\end{equation}
for all $x \in \R/\Z$, then 
\begin{equation}\label{Intro_answer_L_ET_case2}
\int_{\R/\Z}  \ms{L}(x)\, \dvar(x) \leq - \frac{1}{\ms{K}_N(1,1)}\ +  \sum_{\stackrel{\xi \in \R/\Z\,;\, \xi \neq 0}{\ms{B}_{N+1}(e(\xi))=0}}\frac{\wt{\ms{F}}_{\mu}(\xi)}{\ms{K}_{N}(e(\xi), e(\xi))}.
\end{equation} 
Moreover, there is a unique real trigonometric polynomial $\wt{\ms{L}}_{\mu}:\C \to \C$ of degree at most $N$ satisfying \eqref{Intro_eq_L_ET_case2} for which the equality in \eqref{Intro_answer_L_ET_case2} holds.

\smallskip

\item[(ii)] If $\ms{M}:\C \to \C$ is a real trigonometric polynomial of degree at most $N$ such that
\begin{equation}\label{Intro_eq_M_ET_case2}
\ms{F}_{\mu}(x) \leq \ms{M}(x)  
\end{equation}
for all $x \in \R/\Z$, then 
\begin{equation}\label{Intro_answer_M_ET_case2}
\int_{\R/\Z}  \ms{M}(x)\, \dvar(x) \geq \frac{1}{\ms{K}_N(1,1)}\ +  \sum_{\stackrel{\xi \in \R/\Z\,;\, \xi \neq 0}{\ms{B}_{N+1}(e(\xi))=0}}\frac{\wt{\ms{F}}_{\mu}(\xi)}{\ms{K}_{N}(e(\xi), e(\xi))}.
\end{equation} 
Moreover, there is a unique real trigonometric polynomial $\wt{\ms{M}}_{\mu}:\C \to \C$ of degree at most $N$ satisfying \eqref{Intro_eq_M_ET_case2} for which the equality in \eqref{Intro_answer_M_ET_case2} holds.
\end{enumerate}
\end{theorem}

\subsection{Periodic interpolation} Before we proceed to the proofs of Theorems \ref{thm11_ET} and \ref{thm12_ET} we state and prove the periodic version of Proposition \ref{M-L-construction}. Below we keep the notation already used in Section \ref{Interpolation_sec}.

\begin{proposition} \label{ET_lem13}
Let $F$ be a $1$-periodic Laguerre-P\'{o}lya function of exponential type $\tau(F)$. Assume that $F$ has a double zero at the origin and that $F(\alpha_F/2) >0$. Let $\mu$ be a signed Borel measure on $\R$ satisfying {\rm (H1') - (H2) - (H4)}, and let $\ms{F}_{\mu}$ be defined by \eqref{Def_ET_F_mu}.
\begin{itemize}
\item[(i)] The functions $x \mapsto L(F,\mu,x)$ and $x \mapsto M(F,\mu,x)$ belong to $L^1(\R)$.
\smallskip
\item[(ii)] Define the trigonometric polynomials
\begin{equation}\label{ET_Lem13_eq0}
\ms{L}(F, \mu, z) = \sum_{|k| < \frac{\tau(F)}{2\pi}} \widehat{L}(F,\mu,k) \,e(kz)
\end{equation}
and
\begin{equation}\label{ET_Lem13_eq00}
\ms{M}(F, \mu, z) = \sum_{|k| < \frac{\tau(F)}{2\pi}} \widehat{M}(F,\mu,k) \,e(kz).
\end{equation}
Then we have
\begin{equation}\label{ET_Lem13_eq1}
F(x) \,\ms{L}(F, \mu, x) \leq F(x) \,\ms{F}_{\mu}(x) \leq F(x)\, \ms{M}(F, \mu, x)
\end{equation}
for all $x \in \R$. 
\smallskip
\item[(iii)] Moreover, 
\begin{equation}\label{ET_Lem13_eq2}
\ms{L}(F, \mu, \xi)= \ms{F}_{\mu}(\xi) = \ms{M}(F, \mu, \xi) 
\end{equation}
for all $\xi \in \R \setminus \Z$ with $F(\xi) = 0$. At $\xi \in \Z$ we have 
\begin{equation}\label{ET_Lem13_eq3}
\ms{L}(F, \mu, \xi) = \ms{F}_{\mu}(0) \ \ \ \  {\rm and} \ \ \ \ \ms{M}(F, \mu, \xi) = \ms{F}_{\mu}(0) +1.
\end{equation}
\end{itemize}
\end{proposition}

\begin{proof} We have already noted, from \eqref{al-growth}, that $z \mapsto L(F,\mu,z)$ and $z \mapsto M(F,\mu,z)$ are entire functions of exponential type at most $\tau(F)$. From \eqref{L-M-psi-eq} we find that 
\begin{equation*}
|L(F,\mu,x)| + |M(F,\mu,x)| \ll f_{\mu}(x) + \frac{1 + |F(x)|}{1 + x^2}
\end{equation*}
for $x \in \R$. Since $F$ is $1$-periodic, it is bounded on the real line. Hence, in order to prove (i), it suffices to verify that $f_{\mu} \in L^1(\R)$. This is a simple application of Fubini's theorem and conditions (H1') - (H2) - (H4). In fact,
\begin{equation*}
\int_0^{\infty} f_{\mu}(x)\,\dx = \int_0^{\infty}\int_0^{\infty} x e^{-\lambda x}\,\mu(\lambda)\,\dl \,\dx= \int_0^{\infty} \frac{1}{\lambda^2}\,\mu(\lambda)\,\dl < \infty.
\end{equation*}
This establishes (i). The Paley-Wiener theorem implies that the Fourier transforms
\begin{equation*}
\widehat{L}(F,\mu,t) = \int_{-\infty}^{\infty} L(F,\mu,x)\,e(-tx)\,\dx  \ \ \ {\rm and} \ \ \ \widehat{M}(F,\mu,t) = \int_{-\infty}^{\infty} M(F,\mu,x)\,e(-tx)\,\dx 
\end{equation*}
are continuous functions supported in the compact interval $[-\tfrac{\tau(F)}{2\pi}, \tfrac{\tau(F)}{2\pi}]$. By a classical result of Plancherel and P\'{o}lya \cite{PP}, the functions  $z \mapsto L'(F,\mu,z)$ and $z \mapsto M'(F,\mu,z)$ also have exponential type at most $\tau(F)$ and belong to $L^1(\R)$. Therefore, the Poisson summation formula holds as a pointwise identity and we have
\begin{equation}\label{ET_Lem13_eq4}
\ms{L}(F, \mu, x) =  \sum_{|k| < \frac{\tau(F)}{2\pi}} \widehat{L}(F,\mu,k) \,e(kx) = \sum_{n\in \Z} L(F,\mu,x +n)
\end{equation}
and
\begin{equation}\label{ET_Lem13_eq5}
\ms{M}(F, \mu, x)  =  \sum_{|k| < \frac{\tau(F)}{2\pi}} \widehat{M}(F,\mu,k) \,e(kx) = \sum_{n\in \Z} M(F,\mu,x +n).
\end{equation}
Using the fact that 
\begin{equation*}
\ms{F}_{\mu}(x) = \sum_{n \in \Z} f_{\mu}(x + n)
\end{equation*}
for all $x \in \R$, \eqref{ET_Lem13_eq1}, \eqref{ET_Lem13_eq2} and \eqref{ET_Lem13_eq3} now follow from \eqref{ET_Lem13_eq4}, \eqref{ET_Lem13_eq5} and Proposition \ref{M-L-construction}, since $F$ is $1$-periodic. This establishes (ii) and (iii).
\end{proof}

\subsection{Proof of Theorem \ref{thm11_ET}}

Recall that we have normalized our orthonormal polynomials $\varphi_{N+1}$ in order to have $\varphi_{N+1}(1) \in \R$. This implies that $\ms{B}_{N+1}(1)=0$.

\subsubsection{Optimality} If $\ms{L}:\C \to \C$ is a real trigonometric polynomial of degree at most $N$ such that
\begin{equation*}
\ms{L}(x) \leq \ms{F}_{\mu}(x) 
\end{equation*}
for all $x \in \R/\Z$, from Proposition \ref{Sec8_Cor26} we find that
\begin{align}\label{ET_Sec4_pf_eq_cond}
\int_{\R/\Z} \ms{L}(x)\,\d\vartheta(x)    = \ \sum_{\stackrel{\xi \in \R/\Z}{\ms{B}_{N+1}(e(\xi))=0}}\frac{\ms{L}(\xi)}{\ms{K}_{N}(e(\xi), e(\xi))} \leq \frac{\ms{F}_{\mu}(0)}{\ms{K}_N(1,1)}\ +  \sum_{\stackrel{\xi \in \R/\Z\,;\, \xi \neq 0}{\ms{B}_{N+1}(e(\xi))=0}}\frac{\ms{F}_{\mu}(\xi)}{\ms{K}_{N}(e(\xi), e(\xi))}.
\end{align}
This establishes \eqref{Intro_answer_L_ET}. Under (H3) recall that we have 
\begin{equation*}
\ms{F}_{\mu}(0^+) =  \ms{F}_{\mu}(0^-) + 1 =\ms{F}_{\mu}(0) + 1.
\end{equation*}
In an analogous way, using Proposition \ref{Sec8_Cor26}, it follows that if $\ms{M}:\C \to \C$ is a real trigonometric polynomial of degree at most $N$ such that
\begin{equation*}
\ms{F}_{\mu}(x) \leq \ms{M}(x) 
\end{equation*}
for all $x \in \R/\Z$ then
\begin{align*}
\int_{\R/\Z} \ms{M}(x)\,\d\vartheta(x)    = \ \sum_{\stackrel{\xi \in \R/\Z}{\ms{B}_{N+1}(e(\xi))=0}}\frac{\ms{M}(\xi)}{\ms{K}_{N}(e(\xi), e(\xi))} \geq \frac{\ms{F}_{\mu}(0^+)}{\ms{K}_N(1,1)}\ +  \sum_{\stackrel{\xi \in \R/\Z\,;\, \xi \neq 0}{\ms{B}_{N+1}(e(\xi))=0}}\frac{\ms{F}_{\mu}(\xi)}{\ms{K}_{N}(e(\xi), e(\xi))}.
\end{align*}
This establishes \eqref{Intro_answer_M_ET}.

\subsubsection{Existence} Define the trigonometric polynomial 
\begin{equation}\label{def_mathfrak_B}
\mathfrak{B}_{N+1}(z) = \ms{B}_{N+1}(e(z))\, \overline{\ms{B}_{N+1}(e(\ov{z}))}.
\end{equation}
Since the polynomial $\ms{B}_{N+1}$ has degree $N+1$ and has only simple zeros in the unit circle, we conclude that the trigonometric polynomial $\mathfrak{B}_{N+1}$ has degree $N+1$, is nonnegative on $\R$ and has only double real zeros. Since every trigonometric polynomial is of bounded type in the upper half-plane $\U$, it follows by \cite[Problem 34]{B} that $\mathfrak{B}_{N+1}$ is a Laguerre-P\'{o}lya function. 

\smallskip

We now use Proposition \ref{ET_lem13} to construct the functions
\begin{align*}
\ms{L}_{\mu}(z)& := \ms{L}(\mathfrak{B}_{N+1}, \mu, z);\\
\ms{M}_{\mu}(z)& := \ms{M}(\mathfrak{B}_{N+1}, \mu, z).
\end{align*}
Since $\mathfrak{B}_{N+1}$ has exponential type $2\pi(N+1)$ we see from \eqref{ET_Lem13_eq0} and \eqref{ET_Lem13_eq00} that $\ms{L}_{\mu}$ and $\ms{M}_{\mu}$ are trigonometric polynomials of degree at most $N$. Since $\mathfrak{B}_{N+1}$ is nonnegative on $\R$ we conclude from \eqref{ET_Lem13_eq1} that 
\begin{equation*}
\ms{L}_{\mu}(x) \leq \ms{F}_{\mu}(x) \leq \ms{M}_{\mu}(x)
\end{equation*}
for all $x \in \R/\Z$. Moreover, from \eqref{ET_Lem13_eq2}, \eqref{ET_Lem13_eq3} and the quadrature formula given by Proposition \ref{Sec8_Cor26}, we conclude that the equality in \eqref{Intro_answer_L_ET} holds. Under the additional condition (H3), we use \eqref{ET_Sec4_H3_cond} to see that the equality in \eqref{Intro_answer_M_ET} also holds.

\subsubsection{Uniqueness} If $\ms{L}: \C \to \C$ is a real trigonometric polynomial of degree at most $N$ satisfying \eqref{Intro_eq_L_ET} for which the equality in \eqref{Intro_answer_L_ET} holds, from \eqref{ET_Sec4_pf_eq_cond} we must have
$$\ms{L}(\xi) = \ms{F}_{\mu}(\xi) = \ms{L}_{\mu}(\xi)$$
for all $\xi \in \R/\Z$ such that $\ms{B}_{N+1}(e(\xi)) = 0$. Since $\ms{F}_{\mu}$ is differentiable at $\R/\Z - \{0\}$, from \eqref{Intro_eq_L_ET} we must also have
$$\ms{L}'(\xi) = \ms{F}'_{\mu}(\xi) = \ms{L}_{\mu}'(\xi)$$
for all $\xi \in \R/\Z - \{0\}$ such that $\ms{B}_{N+1}(e(\xi)) = 0$. These $2N+1$ conditions completely determine a trigonometric polynomial of degree at most $N$, hence $\ms{L} = \ms{L}_{\mu}$. The proof for the majorant is analogous.

\subsection{Proof of Theorem \ref{thm12_ET}}

\subsubsection{Optimality and uniqueness} These follow exactly as in the proof of Theorem \ref{thm11_ET} using the fact that 
\begin{equation*}
\wt{\ms{F}}_{\mu}(0^\pm) = \pm1.
\end{equation*}

\subsubsection{Existence} We proceed with the construction of the extremal trigonometric polynomials in two steps:

\medskip

\noindent {\it Step 1}. Suppose that $\mu$ satisfies (H4). 

\medskip

In this case we know that
\begin{equation}\label{ET_pf_thm12_sym}
\wt{\ms{F}}_{\mu}(x) = \ms{F}_{\mu}(x) - \ms{F}_{\mu}(-x)
\end{equation}
for all $x \in \R$. With the notation of Proposition \ref{ET_lem13} and $\mathfrak{B}_{N+1}$ given by \eqref{def_mathfrak_B}, we define 
\begin{equation}\label{ET_pf_thm12_def_L}
\wt{\ms{L}}_{\mu}(z) = \ms{L}(\mathfrak{B}_{N+1}(z), \mu, z) - \ms{M}(\mathfrak{B}_{N+1}(-z), \mu, -z)
\end{equation}
and
\begin{equation}\label{ET_pf_thm12_def_M}
\wt{\ms{M}}_{\mu}(z) = \ms{M}(\mathfrak{B}_{N+1}(z), \mu, z) - \ms{L}(\mathfrak{B}_{N+1}(-z), \mu, -z).
\end{equation}
It is clear from \eqref{ET_pf_thm12_sym} and \eqref{ET_Lem13_eq1} that 
\begin{equation}\label{ET_pf_thm12_int_1}
\wt{\ms{L}}_{\mu}(x) \leq \wt{\ms{F}}_{\mu}(x) \leq \wt{\ms{M}}_{\mu}(x)
\end{equation}
for all $x \in \R/\Z$. Moreover, from \eqref{ET_Lem13_eq2} and \eqref{ET_Lem13_eq3} we find that 
\begin{equation}\label{ET_pf_thm12_int_2}
\wt{\ms{L}}_{\mu}(\xi)= \wt{\ms{F}}_{\mu}(\xi) = \wt{\ms{M}}_{\mu}(\xi)
\end{equation}
for all $\xi \in \R/\Z - \{0\}$ such that $\ms{B}_{N+1}(e(\xi)) =0$ and 
\begin{equation}\label{ET_pf_thm12_int_3}
\wt{\ms{L}}_{\mu}(0) = -1 \ \ \ \  {\rm and} \ \ \ \ \wt{\ms{M}}_{\mu}(0) = 1.
\end{equation}
Using the quadrature formula given by Proposition \ref{Sec8_Cor26}, we see that equality holds in \eqref{Intro_answer_L_ET_case2} and \eqref{Intro_answer_M_ET_case2}.

\medskip

\noindent {\it Step 2}. The case of general $\mu$. 

\smallskip

For every $n \in \N$ we define a measure $\mu_n$ given by
\begin{equation*}
\mu_n(\Omega) := \mu \big(\Omega - \tfrac{1}{n}\big),
\end{equation*}
where $\Omega \subset \R$ is a Borel set. Note that $\mu_n$ satisfies (H1') - (H2) - (H3) - (H4). Let $\wt{\ms{F}}_n:= \wt{\ms{F}}_{\mu_n}$, and $\wt{\ms{L}}_n:= \wt{\ms{L}}_{\mu_n}$ and $\wt{\ms{M}}_n:= \wt{\ms{M}}_{\mu_n}$ as in \eqref{ET_pf_thm12_def_L} and \eqref{ET_pf_thm12_def_M}. Since properties \eqref{ET_pf_thm12_int_1}, \eqref{ET_pf_thm12_int_2} and \eqref{ET_pf_thm12_int_3} hold for each $n \in \N$, in order to conclude, it suffices to prove that $\wt{\ms{F}}_n$ converges pointwise to $\wt{\ms{F}}_{\mu}$ and that $\wt{\ms{L}}_n$ and $\wt{\ms{M}}_n$ converge pointwise (passing to a subsequence, if necessary) to trigonometric polynomials $\wt{\ms{L}}_{\mu}$ and $\wt{\ms{M}}_{\mu}$.
\smallskip

Observe first that 
\begin{equation}\label{ET_pf_thm12_dom_conv}
\wt{\ms{F}}_n(x) =  \int_{0}^{\infty} \wt{h}\big(\lambda + \tfrac{1}{n},x\big) \,\mu(\lambda)\,\dl\end{equation}
for all $x \in \R$. From \eqref{Estimate_h_tilde} we see that, for $0 \leq x \leq 1$,
\begin{align}\label{Interesting_estimate}
\left|\wt{h}\big(\lambda + \tfrac{1}{n},x\big)\right| \leq xe^{-\lambda x} + (1-x)e^{-\lambda (1-x)} + r(\lambda),
\end{align}
where $r(\lambda)$ is $O(1)$ for $\lambda <1$ and $O(e^{-\lambda})$ for $\lambda \geq 1$, uniformly in $x \in [0,1]$ and $n \in \N$. For any $x \in [0,1)$ the right-hand side of \eqref{Interesting_estimate} belongs to $L^1(\R^+, \mu(\lambda)\,\dl)$, and therefore we may use dominated convergence in \eqref{ET_pf_thm12_dom_conv} to conclude that $\wt{\ms{F}}_n(x) \to \wt{\ms{F}}_{\mu}(x)$ as $n \to \infty$.

\smallskip

From \eqref{L-M-psi-eq}, \eqref{ET_Lem13_eq4} and \eqref{ET_Lem13_eq5} we find that 
\begin{equation}\label{Sec4_bounded_expression}
\wt{\ms{M}}_n(x) - \wt{\ms{L}}_n(x) = \frac{4\,\mathfrak{B}_{N+1}(x)}{\mathfrak{B}_{N+1}''(0)}\sum_{k \in \Z}\frac{1}{(x+k)^2}.
\end{equation}
for all $x \in \R$. Note that the right-hand side of \eqref{Sec4_bounded_expression} is bounded since $\mathfrak{B}_{N+1}$ is a trigonometric polynomial with a double zero at the integers. Therefore, we arrive at 
\begin{equation*}
-\frac{4\,\mathfrak{B}_{N+1}(x)}{\mathfrak{B}_{N+1}''(0)}\sum_{k \in \Z}\frac{1}{(x+k)^2} +\wt{\ms{F}}_n(x) \leq \wt{\ms{L}}_n(x) \leq \wt{\ms{F}}_n(x) \leq \wt{\ms{M}}_n(x) \leq \wt{\ms{F}}_n(x)  + \frac{4\,\mathfrak{B}_{N+1}(x)}{\mathfrak{B}_{N+1}''(0)}\sum_{k \in \Z}\frac{1}{(x+k)^2}.
\end{equation*}
From \eqref{ET_pf_thm12_dom_conv} and \eqref{Interesting_estimate} we see that 

$$\big|\wt{\ms{F}}_n(x)\big| \leq \int_0^\infty xe^{-\lambda x}\, \mu(\lambda)\,\dl +  \int_0^\infty (1-x)e^{-\lambda (1-x)}\, \mu(\lambda)\,\dl + \int_0^\infty r(\lambda)\, \mu(\lambda)\,\dl \leq C$$
for all $x \in [0,1]$ and $n \in \N$, since each of the first two integrals is a continuous function of $x \in (0,1)$, with finite side limits as $x \to 0$ and $x \to 1$, due to condition (H3) and the computation in \eqref{Intro_H3}. This implies that $\wt{\ms{L}}_n$ and $\wt{\ms{M}}_n$ are uniformly bounded on $\R$. The $2N+1$ Fourier coefficients of $\wt{\ms{L}}_n$ and $\wt{\ms{M}}_n$ are then uniformly bounded on $\R$ and we can extract a subsequence $\{n_k\}$ such that $\wt{\ms{L}}_{n_k} \to \wt{\ms{L}}_{\mu}$ and $\wt{\ms{M}}_{n_k} \to \wt{\ms{M}}_{\mu}$ uniformly in compact sets, where $\wt{\ms{L}}_{\mu}$ and $\wt{\ms{M}}_{\mu}$ are trigonometric polynomials of degree at most $N$. This completes the proof.

\section*{Acknowledgements} 
\noindent E. C. acknowledges support from CNPq - Brazil grants $302809/2011-2$ and $477218/2013-0$, and FAPERJ - Brazil grant $E-26/103.010/2012$.

%\linespread{1.16}

%%%%%%%%%%%%%%%%%%%%%%%%%%%%%%%%%%%%%%%%%%%%%%%%%%%%%%%%%%%%%%%%%%%%%%%%%%%%%%%%%%%%

\end{document}